\documentclass[a4paper,12pt]{article}
\usepackage{amssymb}
\usepackage{mathrsfs}
\usepackage{amsmath}
\usepackage{amsfonts,amsthm,amssymb}
\usepackage{amsfonts}
\usepackage{graphics}
\usepackage{color}
\usepackage{graphicx}
\usepackage{epstopdf}
\usepackage{subfigure}
\usepackage{appendix}

\newtheorem{lemma}{Lemma}[section]
 \newtheorem{theorem}{Theorem}[section]

\begin{document}
\title{Graphs with second largest eigenvalue less than $1/2$}
\author{{Xiaoxia Wu$^{a}$\footnote{Corresponding author: xx\_wu4023@126.com},~~Jianguo Qian$^{b}$,~~Haigen Peng$^{a}$}\\
{\small $^{a}$School of Mathematics and Statistics, Minnan Normal University, Fujian 363000, China}\\
{\small $^{b}$School of Mathematical Sciences, Xiamen University,
Fujian 361005, China}}
\date{} \maketitle
\begin{abstract} We characterize the simple connected graphs with the second
largest eigenvalue less than 1/2, which consists of 13 classes of specific graphs. These 13 classes  hint that $c_{2}\in [1/2, \sqrt{2+\sqrt{5}}]$, where $c_2$ is the minimum real number $c$ for which every real number greater than $c$ is a limit point in the set
of the second largest eigenvalues of the simple connected graphs. We leave it as a problem.
\end{abstract}

\section{ Introduction}
The second largest eigenvalue is one of the particularly concerned
eigenvalues in the theory of graph spectra.  In application, the
second largest eigenvalue has close relations with the hyperbolic
geometry in Lorentz space $R^{p,1}$ \cite{CGSS,NS}, equiangular
lines of elliptic geometry in Euclidean space $R^{p}$
\cite{LS,JTYZZ} and, also the expander in theoretical computer
science \cite{A2}.

As pointed by Cvetkovi\'c and Simi\'c \cite{CS2}, the graphs with
small second largest eigenvalue $\lambda_2$ may have interesting
structural properties.
  In earlier seventies, using the fact that $\lambda_2(H)\leq \lambda_2(G)$ for any induced subgraph $H$ of a graph $G$ (the hereditary property),
  Howes studied the second largest eigenvalue not more than a constant by considering the forbidden induced subgraphs \cite{H}.
  In particular,  Hoffman proposed the problem of characterizing graphs with second largest eigenvalue at most 1, which was considered earlier by Cvetkovi\'c \cite{C2}. Later in
\cite{P}, Petrovi\'c characterized the connected bipartite graphs
with $\lambda_{2}\leq 1$. The trees, unicyclic, bicyclic and
tricyclic graphs with $\lambda_{2}\leq 1$ were determined in
\cite{G,HY,LY,X}, respectively. In \cite{CGK}, the connected graphs
with exactly three distinct eigenvalues and second largest
eigenvalue at most 1 were classified by Cheng et al. Recently, Liu
et al. \cite{LCS} determined all connected $\{K_{1,3},
K_{5}-e\}$-free graphs with $\lambda_{2}\leq 1$.

 In addition to the graphs with $\lambda_{2}\leq 1$, the graphs with $\lambda_{2}$ less than some other smaller constants also receive particular
 attention
 in the literature \cite{CH1,CH3,CGK,CS}. The graphs with  $\lambda_{2}\leq \sqrt{2}-1$ were determined independently by Petrovi\'c \cite{P2} and Li \cite{L}.
 In  \cite{CH1},  Cao and Yuan  characterized the simple graphs with $\lambda_{2}< 1/3$ and further proposed the problem of characterizing the connected graphs
 with $1/3<\lambda_{2}< ({\sqrt{5}-1})/{2}$. Using the hereditary property, this problem was considered by Cvetkovi\'c and Simi\'c  \cite{CS} from the view point of forbidden induced subgraphs. Till now, the problem still remains open in general.

In this paper we characterize the simple connected graphs with the
second largest eigenvalue less than 1/2 (Theorem \ref{main}), which consists of 13 classes of specific graphs. Our result implies that
$1/2$ is a limit point in $A_2$, where $A_2$ is the set of the
second largest eigenvalues of the simple graphs without isolated
vertex. On the other hand, it was shown that $c_2\in
[\sqrt{2}-1,\sqrt{2+\sqrt{5}}]$ \cite{ZC}, where $c_2$ is the minimum real number $c$ for which every real
number greater than $c$ is a limit point of $A_2$.  Our 13 classes of specific graphs hint that $c_{2}\in [1/2, \sqrt{2+\sqrt{5}}]$. We leave it as a problem at the end of the article.

\section{Main results}
Let $G$ be a simple graph of order $n$. We denote by
$\chi(G,\lambda)$ the characteristic polynomial of $G$ and by
$\lambda_i(G)$ the $i$-th largest eigenvalue of the adjacency matrix
of $G$. For two graphs $G$ and $H$, we denote by $G\cup H$ the
disjoint union of $G$ and $H$. The join $G\vee H$ of $G$ and $H$ is
the graph obtained from $G\cup H$ by joining every vertex of $G$ to
every vertex of $H.$ To simplify notation, we write $G=({G_1}
\vee{G_2}) \vee  {G_3}$ by $G={G_1} \vee{G_2} \vee {G_3}$. Further,
we write the union and join of $k$ copies of a graph $G$ by $kG$ and
$k\circ G$, respectively. As usual, we denote by $\overline{G}$ the
complement of $G.$

In the following, much of our proof is a direct calculation, some of
which seems a little tedious and is listed in Appendix. We begin
with some elementary lemmas.
\begin{lemma}\label{2.1} (Cauchy's Interlace Theorem)\cite{F}. Let $A$
be a symmetric $n\times n$ matrix, and $B$ be an $m\times m$
principal submatrix of $A$, for some $m<n$. If the eigenvalues of
$A$ are $\lambda_{1}\geq \lambda_{2}\geq \cdots \geq\lambda_{n}$,
and the eigenvalues of $B$ are $\mu_{1}\geq \mu_{2}\geq \cdots
\geq\mu_{m}$, then for all $1\leq i\leq m$, $\lambda_{i}\geq
\mu_{i}\geq \lambda_{i+n-m}$.
\end{lemma}

By Lemma \ref{2.1}, if $V_0$ is a subset of $k$
vertices of a graph $G$, then for any $i$ with $1\leq i\leq
n-k,\lambda_i(G)\geq \lambda_i(G-V_0)\geq \lambda_{i+k}(G)$.

\begin{lemma}\label{2.2}\cite{W}  If a graph $G$ has no isolated
vertex and ${\overline G}$ is connected, then $G$ contains an
induced subgraph isomorphic to $P_4$ or $2K_2$.
\end{lemma}

\begin{lemma}\label{2.3} \cite{CH1} If a graph $G$  has no isolated
vertex, then $\lambda_2(G)=0$ if and only if $G$ is a complete
$k$-partite graph with $2\leq k\leq n-1.$
\end{lemma}

By Lemma \ref{2.1}, if a graph $H$ satisfies $\lambda_2(H)\geq 1/2$,
then any graph $G$ that contains $H$ as an induced subgraph
satisfies  $\lambda_2(G)\geq 1/2$ too (the hereditary property). By a direct calculation, we
have $\lambda_2(2K_2)=1>1/2$ and $\lambda _2(P_4)=({\sqrt 5 -
1})/{2}
> 1/2$. So the following property follows directly from Lemma
\ref{2.1}.
\begin{lemma}\label{2.4} For any graph $G$, if $\lambda_2(G)< 1/2$, then $G$ contains no induced subgraph isomorphic to $P_4$ or
$2K_2.$
\end{lemma}

\begin{lemma}\label{3.1}
Let $G$ be a connected graph and ${\overline G_1},{\overline G_2},\cdots,{\overline G_k}$ be the components of $\overline{G}$. If $\lambda_2(G)<1/2$, then\\
(i). $\overline{G}$ is not connected, i.e., $k \ge 2$;\\
(ii).  ${G_i}$ contains at least one isolated vertex for every
$i\in\{1,2,\ldots, k\}$.
\end{lemma}
\begin{proof} (i). If $\overline G $ is connected, then by Lemma \ref{2.2},
$G$ contains an induced subgraph isomorphic to ${P_4}$ or $2{K_2}$, a contradiction to Lemma \ref{2.4}.

(ii). Suppose to the contrary  that ${G_i}$ contains no isolated
vertex for some $i\in\{1,2,\ldots,k\}$.  Since ${\overline G_i} $ is
connected, so by Lemma \ref{2.2}, ${G_i}$ contains an induced
subgraph isomorphic to ${P_4}$ or $2{K_2}$. Further, noticing that
$G={G_1} \vee{G _2} \vee \cdots \vee {G _k}$, $G$ contains an
induced subgraph ${P_4}$ or $2{K_2}$ as ${G _i}$ is an induced
subgraph of $G$. This is again a contradiction.
\end{proof}

By Lemma \ref{3.1}, from now on we always write $G$ as the form
$$G={G_1} \vee{G_2} \vee \cdots \vee {G_k},$$
where $k\geq 2$. In addition to $P_4$ and $2K_2$, in the following proposition, we list some other graphs that have the second largest eigenvalue at least $1/2$, which will be used in our forthcoming argument.

\noindent\textbf{Proposition 1.} Let $H_i=X_i\vee K_1$, where $X_i$ is as
listed in the following table. Then for any $i=1,2,\ldots,13$,
$\lambda_2(H_i)\geq 1/2$.

\begin{table}[h] \label{S}
\centering
\begin{tabular}{|l|l|l|l| }
\hline
$X_i$       &$\lambda_2(H_i)$    &  $X_i$       &$\lambda_2(H_i)$  \\
\hline
$X_1={\overline K_2}\cup K_3$    & 0.6784    &  $X_8=K_1\cup({\overline K_2}\vee {\overline K_4} \vee K_2)$ & 0.5010\\
\hline
$X_2={\overline K_2}\cup P_3$     & 0.5293     & $X_9=K_1\cup({\overline K_2}\vee {\overline K_2} \vee {\overline K_2} \vee K_1)$ &0.5030\\
\hline
$X_3={\overline K_3}\cup K_2$     & 0.5720     & $X_{10}=K_1\cup((K_1\cup P_3)\vee K_1)$ &0.5368\\
\hline
$X_4=({\overline K_2}\cup K_2)\vee K_1$ &0.5151 & $X_{11}=K_1\cup((\overline{K}_2\cup K_2)\vee K_1)$ &0.5730\\
\hline
$X_5=(K_1\cup C_3)\vee K_1$       &   0.5451    &   $X_{12}=K_1\cup({\overline K_{1,3}}\vee K_1)$ &0.6818\\
\hline
$X_6=K_1\cup K_5$                 &   0.5135     &  $X_{13}=K_1\cup({\overline P_3}\vee K_2)$ &0.5100\\
\hline
$X_7=K_1\cup({\overline K_3}\vee {\overline K_3} \vee K_2)$ &0.5022 & &\\
\hline
\end{tabular}
\caption{$X_i,i=1,2,\ldots,13$.}
\end{table}
\begin{lemma}\label{3.2} If $\lambda_2(G)<1/2$ and ${G_i}$ is not empty for some $i\in\{1,2,\ldots,k\}$, then\\
(i). ${G_i}$ has exactly one isolated vertex when $k\geq 3$; or\\
(ii). ${G_i}$ has at most two isolated vertices when $k = 2$.
\end{lemma}
\begin{proof}(i). Suppose to the contrary that  ${G_i}$ contains at least two isolated vertices. Since ${G_i}$ is not empty,
${\overline K_2}  \cup {K_2}$ is an induced subgraph of ${G_i}$ and,
hence an induced subgraph of $G.$ Therefore, $({\overline K_2}  \cup
{K_2}) \vee {K_1} \vee {K_1}= {H_4}$ is an induced subgraph of $G$
as $k\geq 3$. By Lemma \ref{2.1} and Table 1, ${\lambda _2}(G) \ge
{\lambda _2}({H_4})
>1/2,$ a contradiction. Further, by Lemma \ref{3.1} (ii),
${G_i}$ has exactly one isolated vertex.

    (ii). To the contrary suppose that ${G_i}$ has at least three isolated vertices. Since
${G_i}$ contains an edge, ${\overline K_3}  \cup {K_2}$ is an
induced subgraph of ${G_i}$  and, hence $({\overline K_3} \cup
{K_2}) \vee {K_1} = {H_3}$ is an induced subgraph of $G$. By Lemma
\ref{2.1} and Table 1, ${\lambda _2}(G) \ge {\lambda _2}({H_3})
> 1/2,$ a contradiction.
\end{proof}

\begin{theorem}\label{3.3} {Let $G={G _1} \vee {G _2}$ be a connected graph of order $n$. If ${G_i}$ is not empty and has exactly two isolated vertices  for some $i\in\{1,2\}$, then  $\lambda_2(G)<1/2$  if and only if $G= ({\overline K_2}\cup K_2)\vee
{\overline K_{n-4}}$.}
\end{theorem}
\begin{proof} If ${G_i}$ has at least two edges, then all edges in
${G _i}$ are in the same component of ${G _i}.$ Otherwise, ${G _i}$
would contain $2{K_2}$ as an induced subgraph and, hence ${\lambda
_2}(G) \ge {\lambda _2}({2{K_2}}) >1/2$ by Lemma \ref{2.1}
and Lemma \ref{2.4}, a contradiction. Therefore,  ${G_i}$ must
contain an induced subgraph isomorphic to ${P_3}$ or ${K_3}$ and,
hence $G$ has $({\overline K_2}  \cup {P_3}) \vee {K_1}= {H_2}$ or
$({\overline K_2}  \cup {K_3}) \vee {K_1}={H_1}$ as an induced
subgraph. By Lemma \ref{2.1} and Table 1, this is again a
contradiction. Thus $G= ({\overline K_2}\cup K_2)\vee {\overline
K_{n-4}}$ by Lemma \ref{3.2}.

Conversely, we prove ${\lambda _2}(({\overline K_2}\cup K_2)\vee
{\overline K_{n-4}})<1/2$. By a direct calculation (see Appendix 1 for details),
we have
$$\chi (G,\lambda ) = {\lambda ^{n-4}}(\lambda  + 1)({\lambda ^3} - {\lambda ^2} - 4(n-4)\lambda  + 2(n-4)).$$
Let $f(\lambda)={\lambda ^3} - {\lambda ^2} - 4(n-4)\lambda  +
2(n-4).$ It is easy to get that $f(-\infty)<0$, $f(0)>0$, $f(1/2)<0$
and $f(+\infty)>0$. Thus the three roots of $f(\lambda)=0$ lie in
$(-\infty,0)$, $(0,1/2)$ and $(1/2,+\infty)$. Therefore
$\lambda_2(G)<1/2$, which completes our proof.
\end{proof}
\begin{lemma}\label{iff}
Let $\lambda_{1}\geq \lambda_{2}\geq \cdots \geq\lambda_{n}$ be all the eigenvalues of a graph $G$. If $G$ is non-empty connected and ${\lambda _2}(G)<1/2$, then
 \begin{equation}\label{eigen2}
\chi \left(G,\frac{1}{2}\right)= \prod\limits_{i = 1}^n \left(\frac{1}{2} -\lambda _i\right) < 0.
\end{equation}
Conversely, if (\ref{eigen2}) holds and  ${\lambda _3}(G)<1/2$, then  ${\lambda _2}(G)<1/2$.
\end{lemma}
\begin{proof} Since $G$ is non-empty connected, ${\lambda _1}(G) \ge {\lambda _1}({K_2}) = 1$. Recall that the largest eigenvalue of a connected graph is simple (Perron-Frobenius theory).
Hence, if ${\lambda _2}(G)<1/2$, then (\ref{eigen2}) holds. Conversely, if (\ref{eigen2}) holds and  ${\lambda _3}(G)<1/2$, then  ${\lambda _2}(G)<1/2$.
\end{proof}
\subsection{${G _i}$ is non-bipartite for some $i\in\{1,2,\ldots,k\}$}
\begin{lemma}\label{3.4} If $\lambda_2(G)<1/2$ and ${G _i}$  is a non-bipartite graph for some $i\in\{1,2,\ldots,k\}$,   then $k = 2$ and $G ={G _i} \vee {\overline K_t}$.
\end{lemma}
\begin{proof}  Since ${G _i}$ is non-bipartite,  ${G _i}$ has
an odd cycle. Let ${C_{2m + 1}}$ be a shortest odd cycle in ${G
_i}$. It is clear that ${C_{2m + 1}}$ is an induced subgraph of
${G_i}.$ If $m >1$, then ${G _i}$ contains ${P_4}$ as an induced
subgraph,  a contradiction. Hence, $m=1$ and ${G _i}$ has ${C_3}$ as
an induced subgraph. By Lemma \ref{3.1} (ii), ${G _i}$ contains an
induced subgraph ${K_1} \cup {C_3}$. Further, if $k \ge 3$, then $G$
has an induced subgraph $({K_1} \cup {C_3}) \vee {K_1} \vee
{K_1}={H_5}$, a contradiction by Table 1. Therefore, $k = 2$ and $G
={G_1} \vee{G_2}.$ Since ${K_1} \cup {C_3}$ is an induced subgraph
of ${G_1}$, if ${G_2} $ is not empty, then ${K_2} = {K_1} \vee
{K_1}$ is an induced subgraph of ${G_2}$ and, hence $({K_1} \cup
{C_3})\vee {K_1} \vee {K_1}={H_5}$  is an induced subgraph of $G$,
again a contradiction. This completes our proof.
\end{proof}
\begin{theorem}\label{3.5} Let $G={G _1} \vee {K_1}$, where ${G_1} $ is a non-bipartite graph. If $\lambda_2(G)<1/2$ then  ${G_1} $ is one of the following graphs:\\
(i). ${K_1}\cup({\overline
K_s}\vee{\overline K_2}\vee {K_2}),$ $2 \le s \le 3$;\\
(ii). ${K_1} \cup ({\overline K_s}  \vee {K_3}),$ $s \ge 1$;\\
(iii). ${K_1} \cup ({\overline K_{{s_1}}} \vee {\overline K_{{s_2}}}
\vee {\overline K_{{s_3}}} ),$ $1 \le {s_3}
\le {s_2} \le {s_1}$;\\
(iv). ${K_1} \cup ( {\overline K_s}\vee
 {\overline P_3}),$ $s \ge 1.$
\end{theorem}
\begin{proof} Let  ${C_{2m + 1}}$ be a shortest odd cycle of ${G_1}$.
 By the same discussion as in the proof of Lemma \ref{3.4},
we have $m = 1$ and, hence ${G_1}$ contains $C_3$ as an induced
subgraph. By Lemma \ref{3.2}, ${G_1}$ has exactly one isolated
vertex, i.e., ${G_1}= {K_1} \cup Q,$ where $Q$ is a non-bipartite
graph without isolated vertex. Further,  $Q$ is connected since $Q$
contains no induced subgraph $2{K_2}$ by Lemma \ref{2.4}.

Since $Q$ is connected and contains no induced subgraph
isomorphic to $P_4$ or $2K_2$,  $\overline Q $ must be
disconnected by Lemma \ref{2.2}. If $\omega (\overline Q ) \ge 5$, then  $Q$
 contains ${K_5}$ as an induced subgraph.  Note that $G = {G _1} \vee {K_1}$
 and ${G_1}= {K_1} \cup Q.$ It follows that $G$ contains an induced subgraph $ ({K_1} \cup {K_5}) \vee {K_1}= H_{6},$
a contradiction. Therefore, $2 \le \omega (\overline Q ) \le 4.$

\noindent{\bf Case 1.} $\omega (\overline Q ) = 2$.

Let ${\overline Q_1}$ and ${\overline Q_2}$ be the two components of
$\overline Q$. Then $Q ={ Q_1}  \vee {Q_2}$. If ${\overline Q_1}$
and ${\overline Q_2}$ are both complete graphs, then $Q$ is a
complete bipartite graph, contradicting that $Q$ is non-bipartite.
Therefore, at least one of ${\overline Q_1}$ and ${\overline Q_2}$,
say ${\overline Q_1}$, is not complete.
 Then ${\overline Q_1}$ contains ${P_3}$ as an induced subgraph and, hence, $\left| {V({\overline Q_1})} \right| \geq3$.

\noindent{\bf Case 1.1.} $\left| {V({\overline Q_1})} \right| = 3$.

In this case, ${\overline Q_1}={P_3}$. If ${\overline Q_2}$ is not
complete, then $\overline Q$ contains ${P_3} \cup {\overline K_2} $
as an induced subgraph and, correspondingly, $Q$ has an induced
subgraph ${\overline P_3}  \vee {K_2}.$  This means that $G$
contains $({K_1} \cup ({\overline P_3}  \vee {K_2})) \vee
{K_1}=H_{13}$ as an induced subgraph since $G=({K_1} \cup Q) \vee
{K_1}$, a contradiction. Therefore, ${\overline Q_2}$ is complete
and, hence $\overline Q = {P_3} \cup {K_s}$ and $Q = {\overline P_3}
\vee {\overline K_s}$ where $s \ge 1.$ This yields ${G_1}= {K_1}
\cup ({\overline P_3}  \vee {\overline K_s})$, $s \ge 1$,  which is
indicated as (iv) in the theorem.

\noindent{\bf Case 1.2.} $\left| {V({\overline Q_1})} \right| \geq
4$.

Since $Q$ contains neither $2K_2$ nor $P_4$ as an induced subgraph,
${\overline Q_1}$ contains neither $\overline{2K_2}=C_4$ nor
$\overline{P_4}=P_4$ as an induced subgraph. Further, notice that
${\overline Q_1}$ is connected and contains ${P_3}$ as an induced
subgraph. We conclude that ${\overline Q_1}$ must contain one of
${K_1 \vee (K_1 \cup K_2)}$, ${K_2 \vee {\overline K_2}}$ and
${K_{1,3}}$ as an induced subgraph. Then $\overline Q $ contains
$({K_1 \vee (K_1 \cup K_2)}) \cup {K_1}$, $({K_2 \vee {\overline
K_2}}) \cup {K_1}$ or ${K_{1,3}} \cup {K_1}$ as an induced subgraph.
Correspondingly, $Q$ contains $(K_1 \cup P_{3}) \vee {K_1}$,
$({{\overline K_2} \cup K_2})  \vee {K_1}$ or ${\overline K_{1,3}}
\vee {K_1}$ as an induced subgraph. Therefore, $G$ contains $({K_1}
\cup ((K_1 \cup P_{3}) \vee {K_1})) \vee {K_1}=H_{10}$, $({K_1} \cup
(({{\overline K_2} \cup K_2})  \vee {K_1})) \vee {K_1}=H_{11}$ or
$({K_1} \cup ({\overline K_{1,3}} \vee {K_1})) \vee {K_1}=H_{12}$ as
an induced subgraph since $G= ({K_1} \cup Q) \vee {K_1}$. This is a
contradiction.

\noindent{\bf  Case 2.} $\omega (\overline Q)\geq 3$.

\noindent\textbf{Claim 1.} {\em  If $\omega (\overline Q ) \geq 3,$ then every
component of $\overline Q $ is a complete graph.}

Let ${\overline Q_1} ,{\overline Q_2},\ldots, {\overline Q_{\omega
(\overline Q )}}$ be the components of $\overline Q$. To the
contrary suppose that ${\overline Q_1}$ is not complete. Then
${\overline Q_1}$ contains an induced subgraph ${P_3}$ and, hence
$\overline Q $ contains ${P_3} \cup {K_1} \cup {K_1}$ as an induced
subgraph. Thus, $Q$ has an induced subgraph ${\overline P_3} \vee
{K_2}$ and, therefore, $G$ contains $({K_1} \cup ({\overline P_3}
\vee {K_2})) \vee {K_1}=H_{13}$ as an induced subgraph, a
contradiction. The claim follows.

\noindent{\bf  Case 2.1}. $\omega (\overline Q ) = 3$.

By Claim 1, the three components of $\overline Q$ are all complete.
We have ${ G_1}={K_1} \cup ({\overline K_{{s_1}}} \vee {\overline
K_{{s_2}}}  \vee {\overline K_{{s_3}}} )$ since ${G_1} = {K_1}\cup
Q$, where $1 \le {s_3} \le {s_2} \le {s_1}.$  This is indicated as
(iii) in the theorem.

\noindent{\bf  Case 2.2}. $\omega (\overline Q ) = 4$.

Let ${\overline Q_1},{\overline Q_2},{\overline Q_3},{\overline
Q_4}$ be the four components of $\overline Q$. By Claim 1,
${\overline Q_1},{\overline Q_2},{\overline Q_3},{\overline Q_4}$
are all complete. If three of ${\overline Q_1},$ ${\overline Q_2},$
${\overline Q_3}$ and ${\overline Q_4}$ are not $K_1$, then
$\overline Q$ contains an induced subgraph ${K_2} \cup {K_2}
\cup{K_2} \cup {K_1}$ and, hence $G$ has an induced subgraph $({K_1}
\cup ({\overline K_2} \vee {\overline K_2} \vee {\overline K_2} \vee
{K_1})) \vee {K_1}=H_{9}$. This is a contradiction. Hence, at most
two of ${\overline Q_1},$ ${\overline Q_2},$ ${\overline Q_3}$ and
${\overline Q_4}$ are not $K_1$.

If $ {\overline Q_1},$ $ {\overline Q_2},$ $ {\overline Q_3},$ $
{\overline Q_4}$ are all $K_1$, then ${G_1} = {K_1} \cup {K_4}={K_1}
\cup ({K_1} \vee {K_3}).$ This is indicated as (ii) in the theorem,
where $s = 1$.

 If exactly one of ${\overline Q_1},$ ${\overline Q_2},$ ${\overline Q_3}$ and ${\overline Q_4}$, say ${\overline Q_1}$, is not $K_1$, then ${\overline Q_1}={K_s}$, where $s \ge
2$. Therefore, $\overline Q={K_s} \cup {K_1} \cup {K_1} \cup
{K_1}={K_s} \cup {\overline K_3}$ and ${G_1}={K_1} \cup ({\overline
K_s}  \vee {K_3})$. This is indicated as (ii) in the theorem, where
$s\geq 2$.

If exactly two of ${\overline Q_1},$ ${\overline Q_2},$ ${\overline
Q_3}$ and ${\overline Q_4}$, say ${\overline Q_1}$ and ${\overline
Q_2}$, are not $K_1$, then ${\overline Q_1}={K_r},$ ${\overline
Q_2}= {K_s},$
 ${Q_3}={K_1}$ and ${Q_4}= {K_1}$, where $r,s \ge 2$. Therefore,  $\overline Q$ contains
an induced subgraph ${K_r} \cup {K_s} \cup {\overline K_2}$ and, hence $Q$ contains an induced
subgraph ${\overline K_r}  \vee {\overline K_s} \vee {K_2}$. Without loss of generality, we assume $r\leq s$.  We claim
 that $r=2$ and $s\leq 3$. Suppose to the contrary that $r\geq 3$ or $s\ge 4$. Since $G= ({K_1} \cup Q) \vee
{K_1}$, then $G$ contains $({K_1} \cup ({\overline K_3}  \vee
{\overline K_3} \vee {K_2}))\vee {K_1}=H_{7}$ or $({K_1} \cup
({\overline K_2} \vee {\overline K_4} \vee {K_2})) \vee{K_1}=H_{8}$
as an induced subgraph. This is a contradiction. As a result, we
have either ${G_1} ={K_1} \cup ({\overline K_2} \vee {\overline K_2}
\vee {K_2})$ or ${G_1} = {K_1} \cup ({\overline K_2}  \vee
{\overline K_3} \vee {K_2}),$ which is indicated as (i) in the
theorem.
\end{proof}

Note that ${G _1} \vee {K_1}$ is an induced subgraph of ${ G _1}
\vee {\overline K_t}$. So by Lemma \ref{2.1} and Lemma \ref{3.4}, if
$G={G _1} \vee {\overline K_t},\lambda_2(G)<1/2$ and ${G _1}$ is
non-bipartite, then  ${G _1}$ must have one of the four forms as
indicated in Theorem \ref{3.5}. In the following we will determine
the exact values of $t$ for the four cases.
\begin{lemma}\label{3.6} Let $G= ({K_1} \cup({\overline K_s}  \vee {\overline K_2}  \vee {K_2})) \vee
{\overline K_t},2 \le s \le 3$. Then $\lambda_2(G)<1/2$ if and only
if $t = 1.$
\end{lemma}
\begin{proof} By a direct calculation (see Appendix 2), we have\\
$\chi (G,\lambda ) =$
\begin{equation}\label{eigen1}
 {\lambda ^{s + t - 1}}(\lambda  + 1)\left({\lambda ^5} - {\lambda ^4} - (st + 5t + 4s + 4){\lambda ^3} - (7st + 6s + 5t){\lambda ^2} - (4st - 4t)\lambda  + 6st\right).
\end{equation}

For specificality, we write $G=G(s,t)$.   If $s = 2$, then by (\ref{eigen1}) we have
$$\chi
\left( {G(2,t),\frac{1}{2}} \right) = {\left( {\frac{1}{2}}
\right)^{t + 1}}\left( {\frac{3}{2}} \right) \times
\frac{1}{{32}}(140t - 145).$$ So by Lemma \ref{iff}, if
$\lambda_2(G(2,t))<1/2$ then $\chi \left( G(2,t),1/2 \right)<0$,
meaning that $t < 2$, i.e., $t=1$. If $s = 3$, then
$$\chi \left( {G(3,t),\frac{1}{2}}
\right) = {\left( {\frac{1}{2}} \right)^{t + 2}}\left( {\frac{3}{2}}
\right) \times \frac{1}{{32}}(208t - 209).$$ Similarly, again by
Lemma  \ref{iff}, if  $\lambda_2(G(3,t))<1/2$  then $t=1$.

Conversely, assume $t=1$. If $s=2$, then by a direct calculation we
have $\lambda_2(G)\approx 0.4968<1/2$ and if $s=3$, then
$\lambda_2(G)\approx 0.4996<1/2$. This completes the proof.
\end{proof}

\begin{lemma}\label{3.7} Let $G=({K_1} \cup
({\overline K_s}  \vee {K_3})) \vee {\overline K_t}$. Then
$\lambda_2(G)<1/2$ if and only if $t = 1$.
\end{lemma}
\begin{proof} By a direct calculation (see Appendix 3), we have
$$\chi (G,\lambda ) = {\lambda ^{s + t - 2}}{(\lambda  + 1)^2}\left({\lambda ^4} - 2{\lambda ^3} - (st + 4t + 3s){\lambda ^2} - (4st - 2t)\lambda  + 3st\right).$$
So by Lemma \ref{iff}, if $\lambda_2(G)<1/2$, then
$$\chi \left(G,\frac{1}{2}\right) = {\left(\frac{1}{2}\right)^{s + t - 2}}{\left(\frac{3}{2}\right)^2}\left(\frac{3}{4}\right)\left(st-s-\frac{1}{4}\right)<0,$$
 meaning that $t=1$.

 We now assume that $t=1$. Let $H={\overline K_s}  \vee {K_3}\vee {\overline K_t}={\overline K_s} \vee {K_1} \vee {K_1}  \vee {K_1}\vee K_1$.  If $s=1$, then $H$ is a complete graph and, hence, $\lambda_2(H)=-1<1/2$. If $s>1$, then by Lemma \ref{2.3}, we have $\lambda_2(H)=0<1/2$. Further, note that $H$ is an induced subgraph of $G$ and has one vertex less than $G$. So by Lemma \ref{2.1}, $\lambda_3(G)\leq \lambda_2(H)<1/2$.  Our lemma follows by Lemma \ref{iff}.
\end{proof}

\begin{lemma}\label{3.8} Let $G=({K_1} \cup ({\overline K_{{s_1}}}  \vee {\overline K_{{s_2}}} \vee
{\overline K_{{s_3}}} )) \vee
{\overline K_t},1 \le {s_3} \le {s_2} \le {s_1}$. \\
(i). If $s_1=s_2=s_3=1$ then $\lambda_2(G)<1/2$ for any $t$ with $t\geq1$; and \\
(ii). if $s_1>1$, then  $\lambda_2(G)<1/2$ if and only if $t
<\frac{\alpha(s_1,s_2,s_3)}{\beta(s_1,s_2,s_3)}$, where
$$\alpha(s_1,s_2,s_3)={16{s_1}{s_2}{s_3} + 4({s_1}{s_2} + {s_2}{s_3} +{s_1}{s_3}) - 1}$$
and
$$\beta(s_1,s_2,s_3)={16{s_1}{s_2}{s_3} - 4({s_1} + {s_2} + {s_3}+1)}.$$
\end{lemma}
\begin{proof} (i) follows directly by a direct calculation. \\
(ii). By a direct calculation (see Appendix 4), we have
\begin{equation}\label{eigen3}
\chi \left(G,\lambda\right)= {\lambda ^{{s_1} + {s_2} + {s_3} + t - 4}}\left({\lambda ^5} - ({s_1}{s_2} + {s_1}{s_3} + {s_2}{s_3} + {s_1}t + {s_2}t + {s_3}t + t){\lambda ^3}\right.$$
$$\left. -2({s_1}{s_2}{s_3} + {s_1}{s_2}t + {s_1}{s_3}t + {s_2}{s_3}t){\lambda ^2} +
        ({s_1}{s_2}t + {s_1}{s_3}t + {s_2}{s_3}t - 3{s_1}{s_2}{s_3}t)\lambda  + 2{s_1}{s_2}{s_3}t\right).
\end{equation}
Hence,
$$\chi\left(G,\frac{1}{2}\right) =\left(\frac{1}{2}\right)^{{s_1} + {s_2} + {s_3} + t -4}\left(\frac{1}{32}\right)(\beta(s_1,s_2,s_3) t-\alpha(s_1,s_2,s_3)).$$
Since $s_1>1$, we have $\beta(s_1,s_2,s_3)>0$. So by Lemma
\ref{iff}, if $\chi(G,1/2)< 0$, then $\beta(s_1,s_2,s_3) t<
\alpha(s_1,s_2,s_3),$  i.e.,  $t
<\alpha(s_1,s_2,s_3)/\beta(s_1,s_2,s_3)$.

 Conversely, let $H={\overline K_{{s_1}}}  \vee {\overline K_{{s_2}}} \vee{\overline K_{{s_3}}}\vee{\overline K_t}$. It is clear that $H$ is an induced subgraph of $G$ and has one vertex less than $G$. The remaining discussion is completely the same as that for Lemma \ref{3.7}.
\end{proof}
\begin{lemma}\label{3.9} Let $G=({K_1} \cup
({\overline K_s}  \vee {\overline P_3}))  \vee {\overline K_t}, s\ge
1$. Then $\lambda_2(G)<1/2$ if and only if $t =1.$
\end{lemma}
\begin{proof}  By a direct calculation (see Appendix 5), we have
$$\chi (G,\lambda ) = {\lambda ^{s + t - 2}}(\lambda  + 1)\left({\lambda ^5} - {\lambda ^4}
- (st + 3s + 4t){\lambda ^3} - (5st - s - 2t){\lambda ^2} +
5st\lambda- st\right).$$

Therefore, $\chi (G,1/2)= {\left(\frac{1}{2}\right)^{s + t -
2}}\left(\frac{1}{2}+1\right)\frac{1}{{32}}(4s(t - 1) - 1)$. By
Lemma \ref{iff}, if $\lambda_2(G)<1/2$ then $\chi (G,1/2)<0$ and,
hence,
 $t=1$.

 Conversely, assume $t=1$. We prove $\lambda_2(G)<1/2$  by induction on $s$.  When $s=1$, one can see that $\lambda_2(G)\approx 0.4897<1/2$.

Write $G$ specifically by $G(s)$ and  assume that
$\lambda_2(G(s))<1/2$ for $s\leq m$, where $m\geq 1$. We note that
$G(m)$ is an induced subgraph of $G(m+1)$ and has one vertex less
than $G(m+1)$. So by the induction hypothesis and Lemma \ref{2.1},
$\lambda_3(G(m+1))\leq\lambda_2(G(m))<1/2$. Again by Lemma \ref{iff}
we have $\lambda_2(G(m+1))<1/2$, which completes the proof.
\end{proof}
By Theorem \ref{3.5} and the lemmas above, we have the following
result.

\begin{theorem}\label{3.10}
Let $G={G_1} \vee{G_2} \vee \cdots \vee {G_k},$ $k \ge 2$, where at
least one of ${G _i}$ is non-bipartite. Then $\lambda_2(G)<1/2$ if
and only if one of the
following holds:\\
(i). $G = ({K_1} \cup ({\overline K_s}  \vee {\overline K_2}
\vee {K_2})) \vee {K_1},2 \le s \le 3$;\\
(ii). $G = ({K_1} \cup ({\overline K_s}  \vee {K_3})) \vee
{K_1},s \ge 1$;\\
(iii). $G = ({K_1} \cup {K_3}) \vee {\overline K_t}, t \ge 1$;\\
(iv). $G =({K_1} \cup ({\overline K_{{s_1}}}  \vee {\overline
K_{{s_2}}}  \vee {\overline K_{{s_3}}} )) \vee {\overline K_t}$,
${s_1}\geq {s_2}\geq {s_3}\geq 1,$ ${s_1} > 1,$ $t
<\frac{\alpha(s_1,s_2,s_3)}{\beta(s_1,s_2,s_3)}$;\\
(v). $G = ({K_1} \cup ({\overline K_s}  \vee {\overline P_3} ))
\vee {K_1}, s \ge 1$.
\end{theorem}

\subsection{${G_i}$ is bipartite for any $i\in\{1,2,\ldots,k\}$}
In this subsection, we consider the case that ${G_i}$ is bipartite
for any $i\in\{1,2,\ldots,k\}$. If ${G_i}$ is empty for any
$i\in\{1,2\ldots,k\}$, then $G$ is a $k$-partite graph and, hence,
$\lambda_2(G)<1/2$. In the following, without loss of generality we
always assume that ${G_1}$ is not empty.
\begin{lemma}\label{b1}
Let $G={G_1} \vee{G_2} \vee \cdots \vee {G_k}$ $(k\geq 2)$. If
${G_i}$ is bipartite for every $i\in\{1,2,\ldots,k\}$ and
$\lambda_2(G)<1/2$, then for any $i\in\{1,2,\ldots,k\}$, ${ G_i}$ is
empty or ${G_i}={\overline K_2}\cup K_2$ or ${G_i}=K_1\cup K_{s,t},
t\geq s\geq 1$.
\end{lemma}
\begin{proof} Assume ${G_i}$ is non-empty. Then by Lemma \ref{3.1},  we may assume that ${G_i}=K_1\cup Q$, where $Q$ is a non-empty graph. If $Q$ is not connected,
then $Q$ must contain $2K_2$ as an induced subgraph, a contradiction
to Lemma \ref{2.4}, or  ${G_i}={\overline K_2}\cup K_2$ by Theorem
\ref{3.3}. If $Q$ is connected and not complete bipartite, then $Q$
must contain $P_4$ as an induced subgraph, again a contradiction to
Lemma \ref{2.4}. Therefore, $Q$ is complete bipartite.
\end{proof}
In the following proposition, we list some particular graphs with the second greatest eigenvalue no less than $1/2$.

\noindent\textbf{Proposition 2.} Let $Y_i$ be as listed in the following
table, in which $T_{s,t}=K_1\cup K_{s,t}$. Then for any
$i=1,2,\ldots,8$, $\lambda_2(Y_i)\geq 1/2$.
\begin{table}[h] \label{S}
\centering
\begin{tabular}{|l|c|l|c| }
\hline
$Y_i$       &$\lambda_2(Y_i)$    &  $Y_i$       &$\lambda_2(Y_i)$  \\
\hline
$Y_1=T_{1,3}\vee T_{1,2}\vee T_{1,2}$  & {0.5031}    &  $Y_5=T_{2,2}\vee T_{1,1}\vee K_2$ & {0.5049}\\
\hline
$Y_2=T_{1,3}\vee T_{1,2}\vee T_{1,1}\vee K_{1,1}$  & {0.5003}  &  $Y_6=T_{2,3}\vee T_{1,1}\vee K_1$ & {0.5152}\\
\hline
$Y_3=T_{1,4}\vee T_{1,2}$  & {0.5065}  &  $Y_7=T_{2,4}\vee T_{1,1}$ & {0.5061}\\
\hline
$Y_4=T_{2,2}\vee T_{1,2}$  & {0.5195}  &  $Y_{8}=T_{3,3}\vee T_{1,1}$ & {0.5130}\\
\hline
\end{tabular}
\caption{$Y_i,i=1,2,\ldots,8$.}
\end{table}
\begin{lemma}\label{b2}
Let ${G_1}=K_1\cup K_{s,t}$ and $t\geq s\geq 3$. If
$\lambda_2(G)<1/2$ then ${G_i}$ is empty, i.e., ${ G_i}={\overline
K_{s_i}}$, for every $i\in\{2,\ldots,k\}$.
\end{lemma}
\begin{proof} If ${G_i}$ is not empty for some $i\geq 2$, then $G$ contains an induced subgraph $(K_1\cup K_{3,3})\vee (K_1\cup K_{1,1})=Y_{8}$. This is a contradiction.
\end{proof}
\begin{lemma}\label{b3}
Let ${G_1}=K_1\cup K_{2,t}$ and $t\geq 2$. If $\lambda_2(G)<1/2$, then  ${G_i}={\overline K_{s_i}}$ for every $i\in\{2,\ldots,k\}$, or one of the following holds:\\
(i).  $t=3,k=2$ and $G=(K_1\cup K_{2,3})\vee (K_1\cup K_{1,1})$;\\
(ii).  $t=2,k\leq 3$ and $G=(K_1\cup K_{2,2})\vee (K_1\cup
K_{1,1})\vee{\overline K_{s_3}},s_3\geq 0$.
\end{lemma}
\begin{proof} Assume that ${G_i}$ is not empty for some $i\geq 2$.

If $t\geq 4$, then $G$ contains an induced subgraph $(K_1\cup
K_{2,4})\vee (K_1\cup K_{1,1})=Y_7$. This is a contradiction. In the
following  we assume that $t\leq 3$. By Lemma \ref{b1},  ${
G_i}=K_1\cup K_{s_i,t_i}$. If $t_i\geq 2$, then $G$ contains an
induced subgraph $(K_1\cup K_{2,2})\vee(K_1\cup K_{1,2})=Y_4$, a
contradiction. This implies that $s_i=t_i=1$ by symmetry.

If $t=3$ and $k\geq 3$, then $G$ contains an induced subgraph
$(K_1\cup K_{2,3})\vee(K_1\cup K_{1,1})\vee K_1=Y_6$, again a
contradiction. Therefore, if $t=3$ then $k=2$ and, hence (i)
follows. If $t=2$ and ${G}_j={G}_l=K_1\cup K_{1,1}$ for some $j,l$
with $j,l\not=i$, then  $G$ contains an induced subgraph $(K_1\cup
K_{2,2})\vee(K_1\cup K_{1,1})\vee K_2=Y_5$, again a contradiction.
Further, notice that  $(K_1\cup K_{2,2})\vee(K_1\cup K_{1,1})\vee
K_2=(K_1\cup K_{2,2})\vee(K_1\cup K_{1,1})\vee K_1\vee K_1$, meaning
that $k\leq 3$. (ii) thereby follows, which completes our proof.
\end{proof}
\begin{lemma}\label{b4}
Let ${G_1}=K_1\cup K_{1,t},t\geq 3$. If $\lambda_2(G)<1/2$, then one
of the following holds:\
(i). ${G_i}=K_1\cup K_{1,1}$ or ${G_i}={\overline K_{s_i}}$ for every $i\in\{2,\ldots,k\}$;\\
(ii). $t=3,{G_2}=K_1\cup K_{1,2}$ and ${G_i}={\overline K_{s_i}}$  for any $i\in\{3,\ldots,k\}$;\\
(iii). $t=3$ and $G=(K_1\cup K_{1,3})\vee (K_1\cup K_{1,2})\vee
(K_1\cup K_{1,1})\vee{\overline K_{s_4}}$.
\end{lemma}
\begin{proof}  By a direct calculation we have
$$\chi(\overline G_1,\lambda)=\chi(K_1\vee(K_1\cup K_t),\lambda)=(\lambda+1)^{t-1}(\lambda^3+(1-t)\lambda^2-(t+1)\lambda+t-1).$$
Write $f(\lambda)=\lambda^3+(1-t)\lambda^2-(t+1)\lambda+t-1$. Since
$t\geq 3$, it is clear that $f(-3/2)>0$. Therefore, the smallest
root of $f(\lambda)$ is smaller than $-3/2$ as
$\lim_{\lambda\rightarrow-\infty}\chi(\overline
G_1,\lambda)=-\infty$. This implies that the smallest eigenvalue of
$\overline G_1$ is smaller than $-3/2$, i.e.,
$\lambda_{n_1}(\overline G_1)\leq-3/2$, where $|\overline G_1|=n_1$.
 Further, by Lemma \ref{3.2},  Lemma \ref{b1} and Lemma \ref{b3}, ${G_i}=K_1\cup K_{1,t_i}$ for any $i\geq 2$, where $t_i\geq 1$.
 With no loss of generality, assume $t_2\geq t_3\geq\cdots\geq t_k$. We show that $t_2< 3$.

Suppose to the contrary that $t_2\geq 3$. By the same discussion as
for $\overline G_1$, we also have $\lambda_{n_2}(\overline
G_2)\leq-3/2$, where $|\overline G_2|=n_2$. Since $\overline G_1$
and $\overline G_2$ are components of ${\overline G}$,
$\lambda_{n_1}(\overline G_1)$ and $\lambda_{n_2}(\overline G_2)$
are also the eigenvalues of ${\overline G}$. This means that the
second smallest eigenvalue of ${\overline G}$ is at most $-3/2$.
Further, for a graph $H$ of order $n$ $(n\geq 2)$ and a positive
integer $k$ $(k\geq 2)$, recall that
$\lambda_k(H)+\lambda_{n-k+1}({\overline H})\geq -1$ (see \cite{CH2} for details).
Therefore,
$$\lambda_2(G)\geq -\lambda_{n-2+1}({\overline G})-1= -\lambda_{n-1}({\overline G})-1\geq \frac{3}{2}-1=1/2.$$
This contradicts our assumption that $\lambda_2(G)<1/2$ and, hence
$t_2< 3$.

If $t_2=2$ and $t\geq 4$, then $G$ contains $Y_3=(K_1\cup
K_{1,4})\vee (K_1\cup K_{1,2})$ as an induced subgraph, a
contradiction. We now assume that $t=3$ and $t_2=2$.

If $t_3=2$, then  $G$ contains $Y_1=(K_1\cup K_{1,3})\vee (K_1\cup
K_{1,2})\vee (K_1\cup K_{1,2})$ as an induced subgraph, a
contradiction. Similarly, if $t_3=t_4=1$, then  $G$ contains
$Y_2=(K_1\cup K_{1,3})\vee (K_1\cup K_{1,2})\vee (K_1\cup
K_{1,1})\vee K_{1,1}$ as an induced subgraph, again a contradiction.
Notice that $K_{1,1}={\overline K_1}\vee{\overline K_1}$. This
completes our proof.
\end{proof}
\begin{lemma}\label{b5}
Let
$$\delta(\lambda,s,t,s_2,\ldots,s_k)=\left(1-\sum\limits_{i=2}^k\frac{s_i}{\lambda+s_i}\right)(\lambda^{3}+(s+t+1)\lambda^{2}+st\lambda-st)-(s+t+1)\lambda^{2}-2st\lambda+st.$$
If ${G_1}=K_1\cup K_{s,t}$ ($s,t\geq 2$) and ${ G_i}={\overline
K}_{s_i}$ for $i\in\{2,\ldots,k\}$, then $\lambda_2(G)<1/2$ if and
only if $\delta(1/2,s,t,s_2,\ldots,s_k)<0$.
\end{lemma}
\begin{proof} By a direct calculation (see Appendix 6), we have
$$\chi(G,\lambda)=\lambda^{s+t+s_2+\cdots+s_k-k-1}\delta(\lambda,s,t,s_2,\ldots,s_k)\prod\limits_{i=2}^k(\lambda+s_i).$$
By Lemma \ref{iff}, if $\lambda_2(G)<1/2$, then $\chi(G,1/2)<0$ and,
hence, $\delta(1/2,s,t,s_2,\ldots,s_k)<0$.

Conversely, assume $\delta(1/2,s,t,s_2,\ldots,s_k)<0$.  Let
$H=K_{s,t}\vee {\overline K}_{s_2}\vee\cdots\vee{\overline
K}_{s_k}$. Then $H$ is an induced complete multipartite subgraph of
$G$ and has one vertex less than $G$. So by Lemma \ref{2.1} and
\ref{2.3}, $\lambda_3(G)\leq\lambda_2(H)=0<1/2$. The lemma follows
by Lemma \ref{iff}.
\end{proof}
\begin{lemma}\label{b6}
If $G=(K_1\cup K_{2,3})\vee (K_1\cup K_{1,1})$ or  $G=(K_1\cup
K_{2,2})\vee (K_1\cup K_{1,1})\vee{\overline K_{s_3}},s_3\geq 0$,
then $\lambda_2(G)<1/2$.
\end{lemma}
\begin{proof} If $G=(K_1\cup K_{2,3})\vee (K_1\cup K_{1,1})$, then by a direct calculation (see Appendix 7), we have
$\lambda_2(G)\approx 0.4974026<0.5.$

Now consider $G=(K_1\cup K_{2,2})\vee (K_1\cup
K_{1,1})\vee{\overline K_{s_3}},s_3\geq 0$. Let  $H=K_{2,2}\vee
(K_1\cup K_{1,1})\vee{\overline K}_{s_3}$. It is clear that
${\overline H}={\overline K}_{2,2}\cup (\overline {K_1\cup
K_{1,1}})\cup K_{s_3}=2K_2\cup K_{1,2}\cup K_{s_3}$. Further, the
smallest eigenvalue of ${\overline H}$ equals the minimum value of
the smallest eigenvalues among the components of  ${\overline H}$,
i.e., $\lambda_{n-1}({\overline
H})=\min\{\lambda_{2}(K_2),\lambda_{3}(K_{1,2}),\lambda_{s_3}(K_{s_3})\}=-\sqrt{2}$.
Therefore, $\lambda_2(H)\leq -\lambda_{n-1}({\overline
H})-1=\sqrt{2}-1<1/2$ (see \cite{CH2} for details). So by Lemma
\ref{2.1} and \ref{2.3}, $\lambda_3(G)\leq\lambda_2(H)<1/2$. The
lemma follows by Lemma \ref{iff}.
\end{proof}
\begin{lemma}\label{b7} \ \\
(i). Let $G=(K_1\cup K_{1,t})\vee(p\circ(K_1\cup
K_{1,1}))\vee{\overline K_{s_{p+2}}}\vee\cdots\vee{\overline
K_{s_{k}}},$ $t\geq 3,$ $p\geq 0.$
Then $\lambda_2(G)<1/2$ if and only if ${\gamma(p,t)}=4tp-10p-4t+1+(2t-5)\sum_{i=p+2}^k\frac{2s_i}{2s_i+1}<0;$\\
(ii). Let $G=(K_1\cup K_{1,3})\vee (K_1\cup K_{1,2})\vee{\overline K_{s_3}}\vee\cdots\vee{\overline K_{s_{k}}}$. Then  $\lambda_2(G)<1/2$ if and only if $\sum_{i=3}^k\frac{2s_i}{2s_i+1}< 3;$\\
(iii). If $G=(K_1\cup K_{1,3})\vee (K_1\cup K_{1,2})\vee (K_1\cup
K_{1,1})\vee{\overline K_{s_4}}$, then   $\lambda_2(G)<1/2$.
\end{lemma}
\begin{proof} (i). By a direct calculation (see Appendix 8), we have
$$\chi(G,\lambda)=\lambda^{\eta +t-1}(\lambda+1)^{p}Q(\lambda),$$
where $\eta=\sum\limits_{i=p+2}^k s_i-k+p+1 $ and

$Q(\lambda)=$
$$\left| {\begin{array}{*{6}{c}}
                1-\sum\limits_{i=p+2}^k\frac{s_i}{\lambda+s_i} &1&1&t&1&2\\
                1&\lambda+1 &1&t&0&0\\
                1&1&\lambda+1 &0&0&0\\
                1&1&0&\lambda+t &0&0\\
                p&0&0&0&\lambda+1&2\\
                p&0&0&0&1&\lambda+1
        \end{array}} \right|
\left| {\begin{array}{*{2}{c}}
                \lambda+1&2\\
                1&\lambda+1
        \end{array}} \right|^{p-1}\prod_{j=p+2}^k(\lambda+s_j).$$

It is clear that $\chi(G,1/2)<0$ if and only if $Q(1/2)<0$. Further,
$$Q(1/2)=\frac{1}{32}\left(4tp-10p-4t+1+(2t-5)\sum_{i=p+2}^k\frac{2s_i}{2s_i+1}\right)\left(\frac{1}{4}\right)^{p-1}\prod_{j=p+2}^{k}(0.5+s_j).$$
So by Lemma \ref{iff}, if $\lambda_2(G)<1/2$ then $Q(1/2)<0$,
meaning that ${\gamma(p,t)}<0$.

Conversely, let $H=K_{1,t}\vee(p\circ(K_1\cup K_{1,1}))\vee{\overline K_{s_{p+2}}}\vee\cdots\vee{\overline K_{s_{k}}}$. One can see that $\lambda_{n-1}({\overline H})=-\sqrt{2}$ and the remaining argument is completely the same as the proof of Lemma \ref{b6}.\\
(ii). By a direct calculation (see Appendix 9), we have
$$\chi(G,\lambda)=\lambda^{\xi+3}\prod_{i=3}^k(\lambda+s_i)R(\lambda),$$
where $\xi=\sum\limits_{i=3}^ks_i-k+2$ and

$$R(\lambda)=\left| {\begin{array}{*{7}{c}}
                1-\sum\limits_{i=3}^k\frac{s_i}{\lambda+s_i} &1&1&3&1&1&2\\
                1&\lambda+1 &1&3&0&0&0\\
                1&1&\lambda+1 &0&0&0&0\\
                1&1&0&\lambda+3 &0&0&0\\
                1&0&0&0&\lambda+1&1&2\\
                1&0&0&0&1&\lambda+1&0\\
                1&0&0&0&1&0&\lambda+2
        \end{array}} \right|.$$
Further, $R(1/2)=\frac{1}{64}(\sum_{i=3}^k\frac{2s_i}{2s_i+1}- 3)$.
Again by Lemma \ref{iff}, if $\lambda_2(G)<1/2$ then $R(1/2)<0$,
meaning that $\sum_{i=3}^k\frac{2s_i}{2s_i+1}< 3$.

Conversely, let $H={\overline K_4}\vee (K_1\cup K_{1,2})\vee{\overline K_{s_3}}\vee\cdots\vee{\overline K_{s_{k}}}$. Since ${\overline K_4}$ is an induced subgraph of $K_1\cup K_{1,3}$, $H$ is an induced subgraph of $G$ and   has one vertex less than $G$. Further, ${\overline H}=K_4\cup (K_1\vee(K_1\cup K_2))\cup K_{s_3}\cup\cdots\cup K_{s_{k}}$ and, hence, $\lambda_{n-1}({\overline H})=\lambda_4(K_1\vee(K_1\cup K_2))>-3/2$. The remaining argument is completely the same as the proof of Lemma \ref{b7}(i). This proves the result.\\
(iii). Let $H={\overline K_{4}}\vee (K_1\cup K_{1,2})\vee (K_1\cup
K_{1,1})\vee{\overline K_{s_4}}$. Then $\lambda_{n-1}({\overline
H})=\lambda_4(K_1\vee(K_1\cup K_2))\approx-1.4812>-3/2$ and the
remaining argument is completely the same as the proof for the
sufficiency of (ii).
\end{proof}

\begin{lemma}\label{b8} \ \\
Let $G=(p\circ(K_1\cup K_{1,2}))\vee(q\circ(K_1\cup
K_{1,1}))\vee{\overline K}_{s_{p+q+1}}\vee\cdots\vee{\overline
K}_{s_k}$, $p+q\geq 1.$ Then $\lambda_2(G)<1/2.$
\end{lemma}
\begin{proof}
Without loss of generality, we may assume that ${ G_1}=K_1\cup
K_{1,2},$ or ${G_2}=K_1\cup K_{1,1}.$ Since ${\overline G_1}$ and
${\overline G_2}$ are components of ${\overline G}$,
$\lambda_{4}({\overline G_1})$ and $\lambda_{3}({\overline G_2})$
are also the eigenvalues of ${\overline G}$. Moreover,
$\lambda_{4}({\overline G_1})\approx -1.4812$ and
$\lambda_{3}({\overline G_2})=-\sqrt{2}$ by routine calculation.
This means that the smallest eigenvalue of ${\overline G}$ is
$\lambda_{4}({\overline G_1})$ or $\lambda_{3}({\overline G_2})$.
Further, for a graph $H$ of order $n$ $(n\geq 2)$ and a positive
integer $k$ $(k\geq 2)$, recall that
$\lambda_k(H)+\lambda_{n-k+2}({\overline H})\leq -1$ \cite{CH2}.
Therefore,
$$\lambda_2(G)\leq -\lambda_{n-2+2}({\overline G})-1= -\lambda_{n}({\overline G})-1< 1/2.$$
\end{proof}
\begin{theorem}\label{main}Let $G$ be a connected graph of order $n$. Then $\lambda(G)<1/2$ if and only if $G$ is one of the following graphs:\\
(1). $({\overline K_2}\cup K_2)\vee{\overline K}_s$, $s \ge 1$;\\
(2). $({K_1} \cup ({\overline K_s}\vee {\overline P_3} ))\vee {K_1}$, $s \ge 1$;\\
(3). $({K_1} \cup ({\overline K_s}  \vee {K_3})) \vee{K_1}$, $s \ge 1$;\\
(4). $({K_1} \cup ({\overline K_s}  \vee {\overline K_2}\vee {K_2})) \vee {K_1}$, $2 \le s \le 3$;\\
(5). $({K_1} \cup ({\overline K_{{s_1}}}  \vee {\overline K_{{s_2}}}
\vee {\overline K_{{s_3}}} )) \vee {\overline K_t}$, ${s_1}\geq
{s_2}\geq {s_3}\geq 1,$ ${s_1} > 1,$ $t
<\frac{\alpha(s_1,s_2,s_3)}{\beta(s_1,s_2,s_3)}$;\\
(6).  $({K_1} \cup {K_3}) \vee {\overline K_t}$, $t \ge 1$;\\
(7). $(p\circ(K_1\cup K_{1,2}))\vee(q\circ(K_1\cup K_{1,1}))\vee{\overline K}_{s_{p+q+1}}\vee\cdots\vee{\overline K}_{s_k}$, $p$, $q\geq 0$;\\
(8). $(K_1\cup K_{1,t})\vee(p\circ(K_1\cup K_{1,1}))\vee{\overline K_{s_{p+2}}}\vee\cdots\vee{\overline K_{s_{k}}}$, $t\geq 3$, $p\geq 0$, ${\gamma(p,t)}<0$;\\
(9). $(K_1\cup K_{1,3})\vee(K_1\cup K_{1,2})\vee{\overline K_{s_3}}\vee\cdots\vee{\overline K_{s_{k}}}$, $\sum_{i=3}^k\frac{2s_i}{2s_i+1}< 3$;\\
(10). $(K_1\cup K_{1,3})\vee (K_1\cup K_{1,2})\vee (K_1\cup K_{1,1})\vee{\overline K}_s$;\\
(11). $(K_1\cup K_{2,2})\vee (K_1\cup K_{1,1})\vee{\overline K_s}$;\\
(12). $(K_1\cup K_{2,3})\vee (K_1\cup K_{1,1})$;\\
(13). $(K_1\cup K_{s,t})\vee{\overline
K_{s_2}}\vee\cdots\vee{\overline K_{s_k}}$, $s,t\geq2$,
$\delta(1/2,s,t,s_2,\ldots,s_k)<0$.
\end{theorem}

{\bf Final remark.} Those graphs from Theorem 2.1 and Theorem 2.3
enable us to see that without the maximum degree hypothesis, these
graphs with $0<\lambda_{2}< 1/2$ have the second eigenvalue
multiplicity at most the constant $5$. Then, we immediately know
that these connected graphs have small second eigenvalue
multiplicity. This is an especially interesting case related to
Theorem 2.2 in \cite{JTYZZ} and Theorem 1.3 in \cite{CH3} .

Let $G_{n}=({\overline K_2}\cup K_2)\vee {\overline K_{n-4}}$. Since
$G_{n}$ is an induced subgraph of $G_{n+1}$,
$\lambda_{2}(G_{n+1})\geq \lambda_{2}(G_{n})$. Therefore, the
sequence $\lambda_{2}(G_{n})$ increases with $n$. Further, by
Theorem \ref{main} (1), $\lambda_{2}(G_{n})<1/2$, meaning that
$\lim_{n\rightarrow \infty}\lambda_{2}(G_{n})$ exists. On the other
hand, in the proof of Theorem \ref{3.3}, we know that
$\lambda_2^3(G_n)-\lambda_2^2(G_n)-4(n-4)\lambda_2(G_n)+2(n-4)=0$
and, hence,
$\lambda_{2}(G_{n})=1/2+(\lambda_{2}^3(G_{n})-\lambda_{2}^2(G_{n}))/(4(n-4))$.
Therefore, $\lim_{n\rightarrow \infty}\lambda_{2}(G_{n})=1/2$, which
means that $1/2$ is a limit point of the second largest eigenvalues
of graphs. Let $A_2$ be the set of the second largest eigenvalues of
simple graphs without isolated vertex and $c_2$ is the minimum real
number $c$ such that every real number greater than $c$ is a limit
point of $A_2$. It was shown that $c_2\in
[\sqrt{2}-1,\sqrt{2+\sqrt{5}}]$ \cite{ZC}. If we could show that
each of the 13 graph classes in Theorem \ref{main} has (if exists)
finite number of limit points, then it would mean that $A_{2}$ is
nowhere dense in the interval $[0, 1/2]$ and, hence, $c_{2}\in [1/2,
\sqrt{2+\sqrt{5}}]$. We leave it as the following problem.

\noindent{\bf Problem.} Is it true that  $c_{2}\in [1/2, \sqrt{2+\sqrt{5}}]$?

\section{Acknowledgment}

The work is supported by the National Natural Science Foundation of
China\, [Grant numbers, 11971406, 12171402], the Natural Science
Foundation of Fujian Province\,[Grant numbers, 2021J06029,
2021J02048, 2021J01978].

\newpage
\begin{center}
{\Large\bf Appendix}
\end{center}

In the following, for a determinant $D$, we use $R_{j}+kR_{i}$ and $C_{j}+kC_{i}$ to denote the addition of $k$ times the $i$-th row to the $j$-th row and $k$ times the $i$-th column
to the $j$-th column of $D$, respectively.

1. (For the proof of Theorem 2.1) Let $G \cong ({\overline K_2} \cup
{K_2}) \vee {\overline K_{n - 4}}.$ Then

    $\chi (G,\lambda ) = {\left| {\begin{array}{*{20}{c}}
                \lambda &0&0&0&{ - 1}& \cdots &{ - 1}\\
                0&\lambda &0&0&{ - 1}& \cdots &{ - 1}\\
                0&0&\lambda &{ - 1}&{ - 1}& \cdots &{ - 1}\\
                0&0&{ - 1}&\lambda &{ - 1}& \cdots &{ - 1}\\
                { - 1}&{ - 1}&{ - 1}&{ - 1}&\lambda & \cdots &0\\
                \vdots & \vdots & \vdots & \vdots & \vdots & \ddots & \vdots \\
                { - 1}&{ - 1}&{ - 1}&{ - 1}&0& \cdots &\lambda
        \end{array}} \right|_{n \times n}} $

By $C_{1}+C_{2},$ $C_{3}+C_{4},$ $C_{5}+\sum\limits_{i=6}^{n}C_{i}$
 and then by $R_{2}-R_{1},$ $R_{4}-R_{3}$ and
$R_{i}-R_{5}$ $(6\leq i\leq n),$ the determinant becomes
    ${\left| {\begin{array}{*{20}{c}}
                    \lambda &0&0&0&{ - n+4}&{ - 1}& \cdots &{ - 1}\\
                    0&\lambda &0&0&0&0& \cdots &0\\
                    0&0&{\lambda  - 1}&{ - 1}&{ - n+4}&{ - 1}& \cdots &{ - 1}\\
                    0&0&0&{\lambda  + 1}&0&0& \cdots &0\\
                    { - 2}&{ - 1}&{ - 2}&{ - 1}&\lambda &0& \cdots &0\\
                    0&0&0&0&0&\lambda & \cdots &0\\
                    \vdots & \vdots & \vdots & \vdots & \vdots & \vdots & \ddots & \vdots \\
                    0&0&0&0&0&0& \cdots &\lambda
            \end{array}} \right|_{n \times n}}$

Then by Laplace expansion along the $i$-th row $(2\leq i\leq n,$ $i \neq
3,5),$ we obtain

 $\begin{array}{l}
\lambda (\lambda + 1) \cdot {\lambda ^{n - 5}}{\left|
{\begin{array}{*{20}{c}}
                    \lambda &0&{ - n+4}\\
                    0&{\lambda  - 1}&{ -n+4}\\
                    { - 2}&{ - 2}&\lambda
            \end{array}} \right|_{3 \times 3}}= {\lambda ^{n-4}}(\lambda  + 1)({\lambda ^3} - {\lambda ^2} - 4(n-4)\lambda  + 2(n-4)).
    \end{array}$

2. (For the proof of Lemma 2.9) Let $G= ({K_1} \cup({\overline K_s}
\vee {\overline K_2}  \vee {K_2})) \vee {\overline K_t}.$ Then

$$\chi (G,\lambda ) = {\left| {\begin{array}{*{20}{c}}
                \lambda &0& \cdots &0&0&0&0&0&{ - 1}& \cdots &{ - 1}\\
                0&\lambda & \cdots &0&{ - 1}&{ - 1}&{ - 1}&{ - 1}&{ - 1}& \cdots &{ - 1}\\
                \vdots & \vdots & \ddots & \vdots & \vdots & \vdots & \vdots & \vdots & \vdots &{}& \vdots \\
                0&0& \cdots &\lambda &{ - 1}&{ - 1}&{ - 1}&{ - 1}&{ - 1}& \cdots &{ - 1}\\
                0&{ - 1}& \cdots &{ - 1}&\lambda &0&{ - 1}&{ - 1}&{ - 1}& \cdots &{ - 1}\\
                0&{ - 1}& \cdots &{ - 1}&0&\lambda &{ - 1}&{ - 1}&{ - 1}& \cdots &{ - 1}\\
                0&{ - 1}& \cdots &{ - 1}&{ - 1}&{ - 1}&\lambda &{ - 1}&{ - 1}& \cdots &{ - 1}\\
                0&{ - 1}& \cdots &{ - 1}&{ - 1}&{ - 1}&{ - 1}&\lambda &{ - 1}& \cdots &{ - 1}\\
                { - 1}&{ - 1}& \cdots &{ - 1}&{ - 1}&{ - 1}&{ - 1}&{ - 1}&\lambda & \cdots &0\\
                \vdots & \vdots & \cdots & \vdots & \vdots & \vdots & \vdots & \vdots & \vdots & \ddots & \vdots \\
                { - 1}&{ - 1}& \cdots &{ - 1}&{ - 1}&{ - 1}&{ - 1}&{ - 1}&0& \cdots &\lambda
        \end{array}} \right|_{n \times n}}$$

By $C_{2}+\sum\limits_{i=3}^{s+1}C_{i}$, $C_{s+2}+C_{s+3},$
$C_{s+4}+C_{s+5},$ $C_{s+6}+\sum\limits_{i=s+7}^{n}C_{i}$, and then
by $R_{i}-R_{2}$ $(3\leq i\leq s+1),$ $R_{s+3}-R_{s+2},$
$R_{s+5}-R_{s+4},$ $R_{i}-R_{s+6}$ $(s+7\leq i\leq n),$ the
determinant becomes
    $$~~~~~~~~~~ {\left| {\begin{array}{*{20}{c}}
                \lambda &0&0& \cdots &0&0&0&0&0&{ - t}&{ - 1}& \cdots &{ - 1}\\
                0&\lambda &0& \cdots &0&{ - 2}&{ - 1}&{ - 2}&{ - 1}&{ - t}&{ - 1}& \cdots &{ - 1}\\
                0&0&\lambda & \cdots &0&0&0&0&0&0&0& \cdots &0\\
                \vdots & \vdots & \vdots & \ddots & \vdots & \vdots & \vdots & \vdots & \vdots & \vdots & \vdots &{}& \vdots \\
                0&0&0& \cdots &\lambda &0&0&0&0&0&0& \cdots &0\\
                0&{ - s}&{ - 1}& \cdots &{ - 1}&\lambda &0&{ - 2}&{ - 1}&{ - t}&{ - 1}& \cdots &{ - 1}\\
                0&0&0& \cdots &0&0&\lambda &0&0&0&0& \cdots &0\\
                0&{ - s}&{ - 1}& \cdots &{ - 1}&{ - 2}&{ - 1}&{\lambda  - 1}&{ - 1}&{ - t}&{ - 1}& \cdots &{ - 1}\\
                0&0&0& \cdots &0&0&0&0&{\lambda  + 1}&0&0& \cdots &0\\
                { - 1}&{ - s}&{ - 1}& \cdots &{ - 1}&{ - 2}&{ - 1}&{ - 2}&{ - 1}&\lambda &0& \cdots &0\\
                0&0&0& \cdots &0&0&0&0&0&0&\lambda & \cdots &0\\
                \vdots & \vdots & \vdots &{}& \vdots & \vdots & \vdots & \vdots & \vdots & \vdots & \vdots &{}& \vdots \\
                0&0&0& \cdots &0&0&0&0&0&0&0& \cdots &\lambda
        \end{array}} \right|_{n \times n}}$$

Then by Laplace expansion along Rows $i$ $(3\leq i\leq n,$ $i \neq
s+2,s+4, s+6),$ we obtain

    $ {\lambda ^{s + t - 1}}(\lambda  + 1){\left| {\begin{array}{*{20}{c}}
                \lambda &0&0&0&{ - t}\\
                0&\lambda &{ - 2}&{ - 2}&{ - t}\\
                0&{ - s}&\lambda &{ - 2}&{ - t}\\
                0&{ - s}&{ - 2}&{\lambda  - 1}&{ - t}\\
                { - 1}&{ - s}&{ - 2}&{ - 2}&\lambda
        \end{array}} \right|}$

    $ = {\lambda ^{s + t - 1}}(\lambda  + 1)[{\lambda ^5} - {\lambda ^4} - (st + 5t + 4s + 4){\lambda ^3} - (7st + 6s + 5t){\lambda ^2} - (4st - 4t)\lambda  + 6st]$.

3. (For the proof of Lemma 2.10) Let $G=({K_1} \cup({\overline K_s}
\vee {K_3})) \vee{\overline K_t}$. Then
$$\chi (G,\lambda ) =
{\left| {\begin{array}{*{20}{c}}
                \lambda &0& \cdots &0&0&0&0&{ - 1}& \cdots &{ - 1}\\
                0&\lambda & \cdots &0&{ - 1}&{ - 1}&{ - 1}&{ - 1}& \cdots &{ - 1}\\
                \vdots & \vdots & \ddots & \vdots & \vdots & \vdots & \vdots & \vdots &{}& \vdots \\
                0&0& \cdots &\lambda &{ - 1}&{ - 1}&{ - 1}&{ - 1}& \cdots &{ - 1}\\
                0&{ - 1}& \cdots &{ - 1}&\lambda &{ - 1}&{ - 1}&{ - 1}& \cdots &{ - 1}\\
                0&{ - 1}& \cdots &{ - 1}&{ - 1}&\lambda &{ - 1}&{ - 1}& \cdots &{ - 1}\\
                0&{ - 1}& \cdots &{ - 1}&{ - 1}&{ - 1}&\lambda &{ - 1}& \cdots &{ - 1}\\
                { - 1}&{ - 1}& \cdots &{ - 1}&{ - 1}&{ - 1}&{ - 1}&\lambda & \cdots &0\\
                \vdots & \vdots &{}& \vdots & \vdots & \vdots & \vdots & \vdots & \ddots & \vdots \\
                { - 1}&{ - 1}& \cdots &{ - 1}&{ - 1}&{ - 1}&{ - 1}&0& \cdots &\lambda
        \end{array}} \right|_{n \times n}}$$

By $C_{2}+\sum\limits_{i=3}^{s+1}C_{i}$,
$C_{s+2}+\sum\limits_{i=s+3}^{s+4}C_{i}$,
$C_{s+5}+\sum\limits_{i=s+6}^{n}C_{i}$, and then by $R_{i}-R_{2}$
$(3\leq i\leq s+1),$ $R_{i}-R_{s+2}$ $(s+3\leq i\leq s+4),$
$R_{i}-R_{s+5}$ $(s+6\leq i\leq n),$ the determinant becomes
    $${\rm{  }}{\left| {\begin{array}{*{20}{c}}
                \lambda &0&0& \cdots &0&0&0&0&{ - t}&{ - 1}& \cdots &{ - 1}\\
                0&\lambda &0& \cdots &0&{ - 3}&{ - 1}&{ - 1}&{ - t}&{ - 1}& \cdots &{ - 1}\\
                0&0&\lambda & \cdots &0&0&0&0&0&0& \cdots &0\\
                \vdots & \vdots & \vdots & \ddots & \vdots & \vdots & \vdots & \vdots & \vdots & \vdots &{}& \vdots \\
                0&0&0& \cdots &\lambda &0&0&0&0&0& \cdots &0\\
                0&{ - s}&{ - 1}& \cdots &{ - 1}&{\lambda  - 2}&{ - 1}&{ - 1}&{ - t}&{ - 1}& \cdots &{ - 1}\\
                0&0&0& \cdots &0&0&{\lambda  + 1}&0&0&0& \cdots &0\\
                0&0&0& \cdots &0&0&0&{\lambda  + 1}&0&0& \cdots &0\\
                { - 1}&{ - s}&{ - 1}& \cdots &{ - 1}&{ - 3}&{ - 1}&{ - 1}&\lambda &0& \cdots &0\\
                0&0&0& \cdots &0&0&0&0&0&\lambda & \cdots &0\\
                \vdots & \vdots & \vdots &{}& \vdots & \vdots & \vdots & \vdots & \vdots & \vdots & \ddots & \vdots \\
                0&0&0& \cdots &0&0&0&0&0&0& \cdots &\lambda
        \end{array}} \right|_{n \times n}}$$

Then by Laplace expansion along Rows $i$ $(3\leq i\leq n,$ $i \neq
s+2,s+5),$ we obtain

    ${\lambda ^{s + t - 2}} \cdot {(\lambda  + 1)^2}{\left| {\begin{array}{*{20}{c}}
                \lambda &0&0&{ - t}\\
                0&\lambda &{ - 3}&{ - t}\\
                0&{ - s}&{\lambda  - 2}&{ - t}\\
                { - 1}&{ - s}&{ - 3}&\lambda
        \end{array}} \right|}$

    $ = {\lambda ^{s + t - 2}}{(\lambda  + 1)^2}[{\lambda ^4} - 2{\lambda ^3} - (st + 4t + 3s){\lambda ^2} - (4st - 2t)\lambda  + 3st]$.

4. (For the proof of Lemma 2.11) Let $G=({K_1} \cup ({\overline
K_{{s_1}}}  \vee {\overline K_{{s_2}}} \vee {\overline K_{{s_3}}} ))
\vee{\overline K_t},1 \le {s_3} \le {s_2} \le {s_1}$. Then
$$\chi (G,\lambda ) = {\left| {\begin{array}{*{20}{c}}
                \lambda &0& \cdots &0&0& \cdots &0&0& \cdots &0&{ - 1}& \cdots &{ - 1}\\
                0&\lambda & \cdots &0&{ - 1}& \cdots &{ - 1}&{ - 1}& \cdots &{ - 1}&{ - 1}& \cdots &{ - 1}\\
                \vdots & \vdots & \ddots & \vdots & \vdots &{}& \vdots & \vdots &{}& \vdots & \vdots &{}& \vdots \\
                0&0& \cdots &\lambda &{ - 1}& \cdots &{ - 1}&{ - 1}& \cdots &{ - 1}&{ - 1}& \cdots &{ - 1}\\
                0&{ - 1}& \cdots &{ - 1}&\lambda & \cdots &0&{ - 1}& \cdots &{ - 1}&{ - 1}& \cdots &{ - 1}\\
                \vdots & \vdots &{}& \vdots & \vdots & \ddots & \vdots & \vdots &{}& \vdots & \vdots &{}& \vdots \\
                0&{ - 1}& \cdots &{ - 1}&0& \cdots &\lambda &{ - 1}& \cdots &{ - 1}&{ - 1}& \cdots &{ - 1}\\
                0&{ - 1}& \cdots &{ - 1}&{ - 1}& \cdots &{ - 1}&\lambda & \cdots &0&{ - 1}& \cdots &{ - 1}\\
                \vdots & \vdots &{}& \vdots & \vdots &{}& \vdots & \vdots & \ddots & \vdots & \vdots &{}& \vdots \\
                0&{ - 1}& \cdots &{ - 1}&{ - 1}& \cdots &{ - 1}&0& \cdots &\lambda &{ - 1}& \cdots &{ - 1}\\
                { - 1}&{ - 1}& \cdots &{ - 1}&{ - 1}& \cdots &{ - 1}&{ - 1}& \cdots &{ - 1}&\lambda & \cdots &0\\
                \vdots & \vdots &{}& \vdots & \vdots &{}& \vdots & \vdots &{}& \vdots & \vdots & \ddots & \vdots \\
                { - 1}&{ - 1}& \cdots &{ - 1}&{ - 1}& \cdots &{ - 1}&{ - 1}& \cdots &{ - 1}&0& \cdots &\lambda
        \end{array}} \right|_{n \times n}}$$

By $C_{2}+\sum\limits_{i=3}^{s_{1}+1}C_{i}$,
$C_{s_{1}+2}+\sum\limits_{i=s_{1}+3}^{s_{1}+s_{2}+1}C_{i}$,
$C_{s_{1}+s_{2}+2}+\sum\limits_{i=s_{1}+s_{2}+3}^{s_{1}+s_{2}+s_{3}+1}C_{i}$,
$C_{s_{1}+s_{2}+s_{3}+2}+\sum\limits_{i=s_{1}+s_{2}+s_{3}+3}^{n}C_{i}$,
 and then by $R_{i}-R_{2}$
$(3\leq i\leq s_{1}+1),$ $R_{i}-R_{s_{1}+2}$ $(s_{1}+3\leq i\leq
s_{1}+s_{2}+1),$ $R_{i}-R_{s_{1}+s_{2}+2}$ $(s_{1}+s_{2}+3\leq i\leq
s_{1}+s_{2}+s_{3}+1),$ $R_{i}-R_{s_{1}+s_{2}+s_{3}+2}$
$(s_{1}+s_{2}+s_{3}+3\leq i\leq n),$ the determinant becomes

    $\footnotesize{\left| {\begin{array}{*{20}{c}}
                \lambda &0&0& \cdots &0&0&0& \cdots &0&0&0& \cdots &0&{ - t}&{ - 1}& \cdots &{ - 1}\\
                0&\lambda &0& \cdots &0&{ - {s_2}}&{ - 1}& \cdots &{ - 1}&{ - {s_3}}&{ - 1}& \cdots &{ - 1}&{ - t}&{ - 1}& \cdots &{ - 1}\\
                0&0&\lambda & \cdots &0&0&0& \cdots &0&0&0& \cdots &0&0&0& \cdots &0\\
                \vdots & \vdots & \vdots & \ddots & \vdots & \vdots & \vdots &{}& \vdots & \vdots & \vdots &{}& \vdots & \vdots & \vdots &{}& \vdots \\
                0&0&0& \cdots &\lambda &0&0& \cdots &0&0&0& \cdots &0&0&0& \cdots &0\\
                0&{ - {s_1}}&{ - 1}& \cdots &{ - 1}&\lambda &0& \cdots &0&{ - {s_3}}&{ - 1}& \cdots &{ - 1}&{ - t}&{ - 1}& \cdots &{ - 1}\\
                0&0&0& \cdots &0&0&\lambda & \cdots &0&0&0& \cdots &0&0&0& \cdots &0\\
                \vdots & \vdots & \vdots &{}& \vdots & \vdots & \vdots & \ddots & \vdots & \vdots & \vdots &{}& \vdots & \vdots & \vdots &{}& \vdots \\
                0&0&0& \cdots &0&0&0& \cdots &\lambda &0&0& \cdots &0&0&0& \cdots &0\\
                0&{ - {s_1}}&{ - 1}& \cdots &{ - 1}&{ - {s_2}}&{ - 1}& \cdots &{ - 1}&\lambda &0& \cdots &0&{ - t}&{ - 1}& \cdots &{ - 1}\\
                0&0&0& \cdots &0&0&0& \cdots &0&0&\lambda & \cdots &0&0&0& \cdots &0\\
                \vdots & \vdots & \vdots &{}& \vdots & \vdots & \vdots &{}& \vdots & \vdots & \vdots &\ddots& \vdots & \vdots & \vdots &{}& \vdots \\
                0&0&0& \cdots &0&0&0& \cdots &0&0&0& \cdots &\lambda &0&0& \cdots &0\\
                { - 1}&{ - {s_1}}&{ - 1}& \cdots &{ - 1}&{ - {s_2}}&{ - 1}& \cdots &{ - 1}&{ - {s_3}}&{ - 1}& \cdots &{ - 1}&\lambda &0& \cdots &0\\
                0&0&0& \cdots &0&0&0& \cdots &0&0&0& \cdots &0&0&\lambda & \cdots &0\\
                \vdots & \vdots & \vdots &{}& \vdots & \vdots & \vdots &{}& \vdots & \vdots & \vdots &{}& \vdots & \vdots & \vdots & \ddots & \vdots \\
                0&0&0& \cdots &0&0&0& \cdots &0&0&0& \cdots &0&0&0& \cdots &\lambda
        \end{array}} \right|_{n \times n}}$

Then by Laplace expansion along Rows $i$ $(3\leq i\leq n,$ $i \neq
s_{1}+2,s_{1}+s_{2}+2,s_{1}+s_{2}+s_{3}+2),$ we obtain

    ${\lambda ^{{s_1} + {s_2} + {s_3} + t - 4}}{\left| {\begin{array}{*{20}{c}}
                \lambda &0&0&0&{ - t}\\
                0&\lambda &{ - {s_2}}&{ - {s_3}}&{ - t}\\
                0&{ - {s_1}}&\lambda &{ - {s_3}}&{ - t}\\
                0&{ - {s_1}}&{ - {s_2}}&\lambda &{ - t}\\
                { - 1}&{ - {s_1}}&{ - {s_2}}&{ - {s_3}}&\lambda
        \end{array}} \right|}$\\
    $\begin{array}{l}
        = {\lambda ^{{s_1} + {s_2} + {s_3} + t - 4}}[{\lambda ^5} - ({s_1}{s_2} + {s_1}{s_3} + {s_2}{s_3} + {s_1}t + {s_2}t + {s_3}t + t){\lambda ^3} -\\ 2({s_1}{s_2}{s_3} + {s_1}{s_2}t + {s_1}{s_3}t + {s_2}{s_3}t){\lambda ^2} +
        ({s_1}{s_2}t + {s_1}{s_3}t + {s_2}{s_3}t - 3{s_1}{s_2}{s_3}t)\lambda  + 2{s_1}{s_2}{s_3}t].
    \end{array}$

5. (For the proof of Lemma 2.12) Let $G=({K_1} \cup ({\overline K_s}
\vee {\overline P_3}))  \vee{\overline K_t}, s\ge 1$. Then
$$\chi (G,\lambda ) = {\left| {\begin{array}{*{20}{c}}
                \lambda &0& \cdots &0&0&0&0&{ - 1}& \cdots &{ - 1}\\
                0&\lambda & \cdots &0&{ - 1}&{ - 1}&{ - 1}&{ - 1}& \cdots &{ - 1}\\
                \vdots & \vdots & \ddots & \vdots & \vdots & \vdots & \vdots & \vdots &{}& \vdots \\
                0&0& \cdots &\lambda &{ - 1}&{ - 1}&{ - 1}&{ - 1}& \cdots &{ - 1}\\
                0&{ - 1}& \cdots &{ - 1}&\lambda &0&0&{ - 1}& \cdots &{ - 1}\\
                0&{ - 1}& \cdots &{ - 1}&0&\lambda &{ - 1}&{ - 1}& \cdots &{ - 1}\\
                0&{ - 1}& \cdots &{ - 1}&0&{ - 1}&\lambda &{ - 1}& \cdots &{ - 1}\\
                { - 1}&{ - 1}& \cdots &{ - 1}&{ - 1}&{ - 1}&{ - 1}&\lambda & \cdots &0\\
                \vdots & \vdots &{}& \vdots & \vdots & \vdots & \vdots & \vdots & \ddots & \vdots \\
                { - 1}&{ - 1}& \cdots &{ - 1}&{ - 1}&{ - 1}&{ - 1}&0& \cdots &\lambda
        \end{array}} \right|_{n \times n}}$$

By $C_{2}+\sum\limits_{i=3}^{s+1}C_{i}$, $C_{s+3}+C_{s+4},$
$C_{s+5}+\sum\limits_{i=s+6}^{n}C_{i}$, and then by $R_{i}-R_{2}$
$(3\leq i\leq s+1),$ $R_{s+4}-R_{s+3},$ $R_{i}-R_{s+5}$ $(s+6\leq
i\leq n),$ the determinant becomes
$${\left| {\begin{array}{*{20}{c}}
                \lambda &0&0& \cdots &0&0&0&0&{ - t}&{ - 1}& \cdots &{ - 1}\\
                0&\lambda &0& \cdots &0&{ - 1}&{ - 2}&{ - 1}&{ - t}&{ - 1}& \cdots &{ - 1}\\
                0&0&\lambda & \cdots &0&0&0&0&0&0& \cdots &0\\
                \vdots & \vdots & \vdots & \ddots & \vdots & \vdots & \vdots & \vdots & \vdots & \vdots &{}& \vdots \\
                0&0&0& \cdots &\lambda &0&0&0&0&0& \cdots &0\\
                0&{ - s}&{ - 1}& \cdots &{ - 1}&\lambda &0&0&{ - t}&{ - 1}& \cdots &{ - 1}\\
                0&{ - s}&{ - 1}& \cdots &{ - 1}&0&{\lambda  - 1}&{ - 1}&{ - t}&{ - 1}& \cdots &{ - 1}\\
                0&0&0& \cdots &0&0&0&{\lambda  + 1}&0&0& \cdots &0\\
                { - 1}&{ - s}&{ - 1}& \cdots &{ - 1}&{ - 1}&{ - 2}&{ - 1}&\lambda &0& \cdots &0\\
                0&0&0& \cdots &0&0&0&0&0&\lambda & \cdots &0\\
                \vdots & \vdots & \vdots &{}& \vdots & \vdots & \vdots & \vdots & \vdots & \vdots &\ddots& \vdots \\
                0&0&0& \cdots &0&0&0&0&0&0& \cdots &\lambda
        \end{array}} \right|_{n \times n}}$$

Then by Laplace expansion along Rows $i$ $(3\leq i\leq n,$ $i \neq
s+2,s+3,s+5),$ we obtain

${\lambda ^{s+t-2}}(\lambda  + 1){\left| {\begin{array}{*{20}{c}}
                \lambda &0&0&0&{ - t}\\
                0&\lambda &{ - 1}&{ - 2}&{ - t}\\
                0&{ - s}&\lambda &0&{ - t}\\
                0&{ - s}&0&{\lambda  - 1}&{ - t}\\
                { - 1}&{ - s}&{ - 1}&{ - 2}&\lambda
        \end{array}} \right|}$

$ = {\lambda ^{s + t - 2}}(\lambda  + 1)[{\lambda ^5} - {\lambda ^4}
- (st + 3s + 4t){\lambda ^3} - (5st - s - 2t){\lambda ^2} +
5st\lambda  - st]$.

6. (For the proof of Lemma 2.17) Let $G=({K_1} \cup { K_{{s,t}}})
\vee {\overline K_{{s_2}}} \vee {\overline K_{{s_3}}}\vee\cdots
\vee{\overline K_{{s_k}}}$. Then $\chi (G,\lambda )= $
$$\footnotesize{\left| {\begin{array}{*{20}{c}}
                \lambda &0& \cdots &0&0& \cdots &0&{ - 1}& \cdots &{ - 1}&{ - 1}& \cdots &{ - 1}& \cdots&{ - 1}& \cdots &{ - 1}\\
                0&\lambda & \cdots &0&{ - 1}& \cdots &{ - 1}&{ - 1}& \cdots &{ - 1}&{ - 1}& \cdots &{ - 1}& \cdots&{ - 1}& \cdots &{ - 1}\\
                \vdots & \vdots & \ddots & \vdots & \vdots &{}& \vdots & \vdots &{}& \vdots & \vdots &{}& \vdots & {} & \vdots &{}& \vdots \\
                0&0& \cdots &\lambda &{ - 1}& \cdots &{ - 1}&{ - 1}& \cdots &{ - 1}&{ - 1}& \cdots &{ - 1}& \cdots&{ - 1}& \cdots &{ - 1}\\
                0&{ - 1}& \cdots &{ - 1}&\lambda & \cdots &0&{ - 1}& \cdots &{ - 1}&{ - 1}& \cdots &{ - 1}& \cdots&{ - 1}& \cdots &{ - 1}\\
                \vdots & \vdots &{}& \vdots & \vdots & \ddots & \vdots & \vdots &{}& \vdots & \vdots &{}& \vdots &{} & \vdots &{}& \vdots \\
                0&{ - 1}& \cdots &{ - 1}&0& \cdots &\lambda &{ - 1}& \cdots &{ - 1}&{ - 1}& \cdots &{ - 1}& \cdots&{ - 1}& \cdots &{ - 1}\\
                { - 1}&{ - 1}& \cdots &{ - 1}&{ - 1}& \cdots &{ - 1}&\lambda & \cdots &0&{ - 1}& \cdots &{ - 1}& \cdots&{ - 1}& \cdots &{ - 1}\\
                \vdots & \vdots &{}& \vdots & \vdots &{}& \vdots & \vdots & \ddots & \vdots & \vdots &{}& \vdots & {} & \vdots &{}& \vdots \\
                { - 1}&{ - 1}& \cdots &{ - 1}&{ - 1}& \cdots &{ - 1}&0& \cdots &\lambda &{ - 1}& \cdots &{ - 1}& \cdots&{ - 1}& \cdots &{ - 1}\\
                { - 1}&{ - 1}& \cdots &{ - 1}&{ - 1}& \cdots &{ - 1}&{ - 1}& \cdots &{ - 1}&\lambda & \cdots &0& \cdots&{ - 1}& \cdots &{ - 1}\\
                \vdots & \vdots &{}& \vdots & \vdots &{}& \vdots & \vdots &{}& \vdots & \vdots & \ddots & \vdots & {} & \vdots &{}& \vdots \\
                { - 1}&{ - 1}& \cdots &{ - 1}&{ - 1}& \cdots &{ - 1}&{ - 1}& \cdots &{ - 1}&0& \cdots &\lambda& \cdots&{ - 1}& \cdots &{ - 1}\\
                \vdots & \vdots &{}& \vdots & \vdots &{}& \vdots & \vdots &{}& \vdots & \vdots & {} & \vdots & \ddots & \vdots &{}& \vdots \\
                { - 1}&{ - 1}& \cdots &{ - 1}&{ - 1}& \cdots &{ - 1}&{ - 1}& \cdots &{ - 1}&{ - 1} & \cdots &{ - 1}& \cdots&{\lambda}& \cdots &{ 0}\\
                \vdots & \vdots &{}& \vdots & \vdots &{}& \vdots & \vdots &{}& \vdots & \vdots & {} & \vdots & {} & \vdots &\ddots& \vdots \\
                { - 1}&{ - 1}& \cdots &{ - 1}&{ - 1}& \cdots &{ - 1}&{ - 1}& \cdots &{ - 1}&{ - 1}& \cdots &{ - 1}& \cdots&{0}& \cdots &{ \lambda}
        \end{array}} \right|_{n \times n}}$$

By $C_{2}+\sum\limits_{i=3}^{s+1}C_{i}$,
$C_{s+2}+\sum\limits_{i=s+3}^{s+t+1}C_{i}$,
$C_{s+t+2}+\sum\limits_{i=s+t+3}^{s+t+s_{2}+1}C_{i}$,
$C_{s+t+s_{2}+2}+\sum\limits_{i=s+t+s_{2}+3}^{s+t+s_{2}+s_{3}+1}C_{i}$,
 $\cdots,$
$C_{s+t+s_{2}+\cdots+s_{k-1}+2}+\sum\limits_{i=s+t+3+\sum\limits_{j=2}^{k-1}s_{j}}^{n}C_{i}$
, and then by operations $R_{i}-R_{2}$ $(3\leq i\leq s+1),$
$R_{i}-R_{s+2}$ $(s+3\leq i\leq s+t+1),$ $R_{i}-R_{s+t+2}$
$(s+t+3\leq i\leq s+t+s_{2}+1)$, $\cdots$,
$R_{i}-R_{s+t+s_{2}+\cdots+s_{k-1}+2}$
$(s+t+3+\sum\limits_{j=2}^{k-1}s_{j}\leq i\leq n),$ the determinant
becomes
    $$\footnotesize{\left| {\begin{array}{*{20}{c}}
                \lambda &0&0& \cdots &0&0&0& \cdots &0&{ -s_{2}}&{ - 1}& \cdots &{ - 1}& \cdots &{ -s_{k}} &{ - 1}& \cdots&{ - 1}\\
                0&\lambda &0& \cdots &0&{ - t}&{ - 1}& \cdots &{ - 1}&{ -s_{2}}&{ - 1}& \cdots &{ - 1}& \cdots &{ -s_{k}} &{ - 1}& \cdots&{ - 1}\\
                0&0&\lambda & \cdots &0&0&0& \cdots &0&0&0& \cdots &0& \cdots &0 &0& \cdots &0\\
                \vdots & \vdots & \vdots & \ddots & \vdots & \vdots &\vdots & {} & \vdots & \vdots& \vdots & {} & \vdots & {} & \vdots & \vdots &{}& \vdots \\
                0&0&0& \cdots &\lambda &0&0& \cdots &0&0&0  & \cdots & 0& \cdots &0&0 &\cdots& 0\\
                0&{ -s} &{ - 1}& \cdots &{ - 1}&{ \lambda}&{ 0}& \cdots &{ 0}&{ -s_{2}}&{ - 1}& \cdots &{ - 1}& \cdots &{ -s_{k}} &{ - 1}& \cdots&{ - 1}\\
                0&0&0 & \cdots &0&0 &\lambda & \cdots &0&0&0& \cdots &0& \cdots &0 &0& \cdots &0\\
                \vdots & \vdots & \vdots & {} & \vdots & \vdots &\vdots & \ddots & \vdots & \vdots& \vdots & {} & \vdots & {} & \vdots & \vdots &{}& \vdots \\
                0&0&0& \cdots &0 &0&0& \cdots & \lambda &0&0  & \cdots & 0& \cdots &0&0 &\cdots& 0\\
                { -1}&{ -s} &{ - 1}& \cdots &{ - 1}&{ { -t}}&{- 1}& \cdots &{ - 1}&{ \lambda}&{ 0}& \cdots &{ 0}& \cdots &{ -s_{k}} &{ - 1}& \cdots&{ - 1}\\
                0&0&0 & \cdots &0&0 &0 & \cdots &0&0&\lambda& \cdots &0& \cdots &0 &0& \cdots &0\\
                \vdots & \vdots & \vdots & {} & \vdots & \vdots &\vdots & {} & \vdots & \vdots& \vdots & \ddots & \vdots & {} & \vdots & \vdots &{}& \vdots \\
                0&0&0& \cdots &0 &0&0& \cdots & 0 &0&0& \cdots & \lambda& \cdots &0&0 &\cdots& 0\\
                \vdots & \vdots & \vdots & {} & \vdots & \vdots &\vdots & {} & \vdots & \vdots& \vdots & {} & \vdots & \ddots & \vdots & \vdots &{}& \vdots \\
                { -1}&{ -s} &{ - 1}& \cdots &{ - 1}&{ { -t}}&{- 1}& \cdots &{ - 1}&{ -s_{2}}&{-1}& \cdots &{-1}& \cdots &{ \lambda} &{ 0}& \cdots&{ 0}\\
                0&0&0 & \cdots &0&0 &0 & \cdots &0&0&\lambda& \cdots &0& \cdots &0 &\lambda& \cdots &0\\
                \vdots & \vdots & \vdots & {} & \vdots & \vdots &\vdots & {} & \vdots & \vdots& \vdots & {} & \vdots & {} & \vdots & \vdots &\ddots& \vdots \\
                0&0&0& \cdots &0 &0&0& \cdots & 0 &0&0& \cdots & \lambda& \cdots &0&0 &\cdots& \lambda
        \end{array}} \right|_{n \times n}}$$

Then by Laplace expansion along Rows $i$ $(3\leq i\leq n),$ $i \neq
2+s,2+s+t,2+s+t+s_{2},2+s+t+s_{2}+s_{3},\cdots,
2+s+t+s_{2}+\cdots+s_{k-1}),$ we obtain

$${\left| {\begin{array}{*{20}{c}}
                \lambda &0&0&{ -s_{2}}&\cdots&{ -s_{k}}\\
                0&\lambda &{ - t}&{ - {s_2}}&\cdots&{ -s_{k}}\\
                0&{ - s}&\lambda &{ - {s_2}}&\cdots&{ -s_{k}}\\
                {-1}&{ - s}&{ - t}&\lambda &\cdots&{-s_{k}}\\
                \vdots & \vdots & \vdots & \vdots & \ddots & \vdots\\
                { - 1}&{ - s}&{ -t}&{ - {s_2}}&\cdots&\lambda
        \end{array}} \right|_{(k+2) \times (k+2)}}={\left| {\begin{array}{*{20}{c}}
                1&0 &0&0&0&\cdots&0\\
                1&\lambda &0&0&{ -s_{2}}&\cdots&{ -s_{k}}\\
                1&0&\lambda &{ - t}&{ - {s_2}}&\cdots&{ -s_{k}}\\
                1&0&{ - s}&\lambda &{ - {s_2}}&\cdots&{ -s_{k}}\\
                1&{-1}&{ - s}&{ - t}&\lambda &\cdots&{-s_{k}}\\
                \vdots&\vdots & \vdots & \vdots & \vdots & \ddots & \vdots\\
                1&{ - 1}&{ - s}&{ -t}&{ - {s_2}}&\cdots&\lambda
        \end{array}} \right|_{(k+3) \times (k+3)}}$$

By $C_{2}+C_{1},$ $C_{3}+sC_{1},$
 $C_{4}+tC_{1},$
 $C_{3+i}+s_{i}C_{1}$
$(2\leq i\leq k),$ and then by
$R_{1}-\sum\limits_{i=2}^{k}\frac{s_{i}}{\lambda+s_{i}}R_{i+3}$
$(\lambda\neq -s_{i})$,

$${\left| {\begin{array}{*{20}{c}}
                1-\displaystyle\sum_{i=2}^{k}\frac{s_i}{\lambda+s_i}&1 &s&t&0&\cdots&0\\
                1&{\lambda+1} &s&t&0&\cdots&0\\
                1&1&{\lambda+s} &0&0&\cdots&0\\
                1&1&0&{\lambda+t}&0&\cdots&0\\
                1&0&0&0&{\lambda+s_{2}} &\cdots&0\\
                \vdots&\vdots & \vdots & \vdots & \vdots & \ddots & \vdots\\
                1&0&0&0&0&\cdots&{\lambda+s_{k}}
        \end{array}} \right|_{(k+3) \times (k+3)}}$$

Then by Laplace expansion along Columns $i$ $(5\leq i\leq k+3),$

$\chi (G,\lambda ) =\lambda^{(s+t+s_2+s_3+\cdots+s_k-k-1)}
(\displaystyle\prod_{i = 2}^{k}(\lambda+s_i)){\left|
{\begin{array}{*{20}{c}}
     1-\displaystyle\sum_{i=2}^{k}\frac{s_i}{\lambda+s_i}&1&s&t\\
        1&\lambda+1&s&t\\
        1&1&\lambda+s&0\\
        1&1&0&\lambda+t
            \end{array}} \right|.}$

Let $\delta(\lambda,s,t,s_2,\ldots,s_k)={\left|
{\begin{array}{*{20}{c}}
     1-\displaystyle\sum_{i=2}^{k}\frac{s_i}{\lambda+s_i}&1&s&t\\
        1&\lambda+1&s&t\\
        1&1&\lambda+s&0\\
        1&1&0&\lambda+t
            \end{array}} \right|}$
            $$=\left(1-\sum\limits_{i=2}^k\frac{s_i}{\lambda+s_i}\right)(\lambda^{3}+(s+t+1)\lambda^{2}+st\lambda-st)-(s+t+1)\lambda^{2}-2st\lambda+st.$$

7. (For the proof of Lemma 2.18) Let $G=(K_1\cup K_{2,3})\vee
(K_1\cup K_{1,1})$. Then
$$\chi(G,\lambda)={\left| {\begin{array}{*{20}{c}}
                {\lambda}&0 &0&0&0&0 &{-1}&{-1}&{-1}\\
                0&{\lambda} &0&{-1}&{-1}&{-1}&{-1}&{-1}&{-1}\\
                0& 0&{\lambda}&{-1}&{-1}&{-1}&{-1}&{-1}&{-1}\\
                0& {-1}&{-1}&{\lambda}&0&0&{-1}&{-1}&{-1}\\
                0& {-1}&{-1}&0&{\lambda}&0&{-1}&{-1}&{-1}\\
                0& {-1}&{-1}&0&0&{\lambda}&{-1}&{-1}&{-1}\\
                {-1}& {-1}&{-1}&{-1}&{-1}&{-1}&{\lambda}&{0}&{0}\\
                {-1}& {-1}&{-1}&{-1}&{-1}&{-1}&{0}&{\lambda}&{-1}\\
                {-1}& {-1}&{-1}&{-1}&{-1}&{-1}&{0}&{-1}&{\lambda}
        \end{array}} \right|}.$$
By MATLAB, we have $\lambda_{2}(G)\approx 0.4974<0.5$.

Let $G=(K_1\cup K_{2,2})\vee (K_1\cup K_{1,1})\vee{\overline
K_{s_3}}$. Then

$$\chi(G,\lambda)={\left| {\begin{array}{*{20}{c}}
                {\lambda}&0 &0&0&0&{-1} &{-1}&{-1}&{-1}&\cdots&{-1}\\
                0&{\lambda} &0&{-1}&{-1}&{-1}&{-1}&{-1}&{-1}&\cdots&{-1}\\
                0& 0&{\lambda}&{-1}&{-1}&{-1}&{-1}&{-1}&{-1}&\cdots&{-1}\\
                0& {-1}&{-1}&{\lambda}&0&{-1}&{-1}&{-1}&{-1}&\cdots&{-1}\\
                0& {-1}&{-1}&0&{\lambda}&{-1}&{-1}&{-1}&{-1}&\cdots&{-1}\\
                {-1}& {-1}&{-1}&{-1}&{-1}&{\lambda}&0&0&{-1}&\cdots&{-1}\\
                {-1}& {-1}&{-1}&{-1}&{-1}&0&{\lambda}&{-1}&{-1}&\cdots&{-1}\\
                {-1}& {-1}&{-1}&{-1}&{-1}&0&{-1}&{\lambda}&{-1}&\cdots&{-1}\\
                {-1}&{-1}&{-1}&{-1}&{-1}&{-1}&{-1}&{-1}&{\lambda}&\cdots&{0}\\
                \vdots & \vdots & \vdots & \vdots & \vdots & \vdots & \vdots& \vdots & \vdots & \ddots & \vdots\\
                {-1}& {-1}&{-1}&{-1}&{-1}&{-1}&{-1}&{-1}&{0}&\cdots&{\lambda}
        \end{array}} \right|_{n \times n}}.$$

By $C_{i}+C_{i+1}$ $(i=2,4,7),$ $C_{9}+\sum\limits_{i=10}^{n}C_{i}$
, and then by $R_{i+1}-R_{i}$ $(i=2,4,7),$ $R_{i}-R_{9}$ $(10\leq
i\leq n),$ the determinant becomes

$${\left| {\begin{array}{*{20}{c}}
                {\lambda}&0 &0&0&0&{-1} &{-2}&{-1}&{-s_{3}}&{-1}&\cdots&{-1}\\
                0&{\lambda} &0&{-2}&{-1}&{-1} &{-2}&{-1}&{-s_{3}}&{-1}&\cdots&{-1}\\
                0&0&{\lambda} &{0}&{0}&{0}&{0}&{0}&{0}&{0}&\cdots&{0}\\
                0& {-2}&{-1}&{\lambda}&{0}&{-1}&{-2}&{-1}&{-s_{3}}&{-1}&\cdots&{-1}\\
                0&0&{0}&{0}&{\lambda} &{0}&{0}&{0}&{0}&{0}&\cdots&{0}\\
                {-1}& {-2}&{-1}&{-2}&{-1}&{\lambda}&{0}&{0}&{-s_{3}}&{-1}&\cdots&{-1}\\
                {-1}& {-2}&{-1}&{-2}&{-1}&{0}&{\lambda-1}&{-1}&{-s_{3}}&{-1}&\cdots&{-1}\\
                 0 &0&{0}&{0}&{0}&{0}&{0}&{\lambda+1}&{0}&{0}&\cdots&{0}\\
                {-1}& {-2}&{-1}&{-2}&{-1}&{-1}& {-2}&{-1}&{\lambda}&0&\cdots&{0}\\
                 0 &0&{0}&{0}&{0}&{0}&{0}&{0}&{0}&{\lambda}&\cdots&{0}\\
                \vdots & \vdots & \vdots & \vdots & \vdots & \vdots & \vdots& \vdots & \vdots  & \vdots & \ddots & \vdots\\
                 0 &0&{0}&{0}&{0}&{0}&{0}&{0}&{0}&{0}&\cdots&{\lambda}
        \end{array}} \right|_{n \times n}}.$$

Then by Laplace expansion along Rows $i$ $(3\leq i\leq
n,i\neq4,6,7,9),$

$\chi(G,\lambda)={(\lambda  + 1)\lambda ^{s_{3}+1}}{\left|
{\begin{array}{*{20}{c}}
                \lambda &0&0&{ - 1}&{ -2}&{ -s_{3}}\\
                0&\lambda&{ - 2} &{ - 1}&{ - 2}&{ -s_{3}}\\
                0&{ - 2}&\lambda&{ - 1} &{ - 2}&{ -s_{3}}\\
                { - 1}&{ - 2}&{ - 2}&{\lambda }&0&{-s_{3}}\\
                { - 1}&{ - 2}&{ - 2}&0&{\lambda-1 }&{-s_{3}}\\
               { - 1}&{ - 2}&{ - 2}&{-1 }&{ - 2}&{\lambda}
        \end{array}} \right|}$

8. (For the proof of Lemma 2.19(i)) Let $G=(K_1\cup
K_{1,t})\vee(p\circ(K_1\cup K_{1,1}))\vee{\overline
K_{s_{p+2}}}\vee\cdots\vee{\overline K_{s_{k}}}.$ Then $$\chi
(G,\lambda )={\left| {\begin{array}{*{20}{c}}
                 A_{11}& A_{12} \\
                  A_{21} &  A_{22}
                       \end{array}} \right|}$$

where

\begin{displaymath}
A_{11}=\left(
\begin{array}{ccccccccccccc}
                \lambda &0&0&0& \cdots &0&{ - 1}&{ - 1}&{ - 1}&\cdots&{ - 1}&{ - 1}&{ - 1}\\
                0&\lambda & { - 1}& { - 1}& \cdots & { - 1}&{ - 1}& { - 1} &{ - 1}& \cdots&{ - 1}&{ - 1}&{ - 1}\\
                0& { - 1}&\lambda & {0}& \cdots & {0}&{ - 1}& { - 1} &{ - 1}& \cdots&{ - 1}&{ - 1}&{ - 1}\\
                0& { - 1}& {0}&\lambda & \cdots & {0}&{ - 1}& { - 1} &{ - 1}& \cdots&{ - 1}&{ - 1}&{ - 1}\\
                \vdots & \vdots & \vdots & \vdots& \ddots & \vdots & \vdots & \vdots& \vdots &\vdots& \vdots &\vdots& \vdots \\
                0&{ - 1}&0&0& \cdots &\lambda &{ - 1}& { - 1} &{ - 1}& \cdots&{ - 1}&{ - 1}&{ - 1}\\
                { - 1}& { - 1}&{ - 1}&{ - 1}& \cdots &{ - 1}&\lambda&0&0 & \cdots &{ - 1}&{ - 1}&{ - 1}\\
                { - 1}& { - 1}&{ - 1}&{ - 1}& \cdots &{ - 1}&0&\lambda&{ - 1} & \cdots &{ - 1}&{ - 1}&{ - 1}\\
                { - 1}& { - 1}&{ - 1}&{ - 1}& \cdots &{ - 1}&0&{ - 1}&\lambda & \cdots &{ - 1}&{ - 1}&{ - 1}\\
                \vdots & \vdots & \vdots & \vdots& \vdots & \vdots & \vdots& \vdots & \vdots& \ddots  &\vdots& \vdots &\vdots\\
                 { - 1}& { - 1}&{ - 1}&{ - 1}& \cdots &{ - 1}&{ - 1}&{ - 1}&{ - 1}& \cdots&\lambda&0&0 \\
                { - 1}& { - 1}&{ - 1}&{ - 1}& \cdots &{ - 1}&{ - 1}&{ - 1}&{ - 1}& \cdots &0&\lambda&{ - 1}\\
                { - 1}& { - 1}&{ - 1}&{ - 1}& \cdots &{ - 1}&{ - 1}&{ - 1}&{ - 1}&\cdots &0&{ - 1}&\lambda\\
                 \end{array} \right)_{(t+3p+2) \times (t+3p+2)},
\end{displaymath}

\begin{displaymath}
A_{22}=\left(
\begin{array}{ccccccc}
                \lambda& \cdots &0 & \cdots &{ - 1}& \cdots &{ - 1}\\
                 \vdots & \ddots  &\vdots& {}& \vdots & {} & \vdots\\
               0& \cdots&\lambda & \cdots&{ - 1}& \cdots &{ - 1}\\
                \vdots & {} & \vdots & \ddots  &\vdots& {} & \vdots \\
               { - 1}& \cdots &{ - 1}& \cdots &\lambda& \cdots &0\\
                 \vdots & {} & \vdots &{}& \vdots &  \ddots  & \vdots\\
               { - 1}& \cdots &{ - 1}& \cdots &0 & \cdots&\lambda\\
                 \end{array} \right)_{(n-t-3p-2) \times (n-t-3p-2)},
\end{displaymath}

and $A_{ij}$ $(1\leq i,j\leq 2, i\neq j)$ denotes the matrix each of
whose entries is $-1.$

By $C_{3}+\sum\limits_{i=4}^{t+2}C_{i}$, $C_{t+3i+4}+C_{t+3i+5}$
$(0\leq i\leq p-1),$
$C_{t+3p+3}+\sum\limits_{i=t+3p+4}^{t+3p+s_{p+2}+2}C_{i}$,
 $C_{t+3p+s_{p+2}+3}+\sum\limits_{i=t+3p+s_{p+2}+4}^{t+3p+s_{p+2}+s_{p+3}+2}C_{i}$
, $\cdots,$
$C_{t+3p+s_{p+2}+\cdots+s_{k-1}+3}+\sum\limits_{i=t+3p+4+\sum\limits_{j=p+2}^{k-1}s_{j}}^{n}C_{i}$,
 and then by $R_{i}-R_{3}$ $(4\leq i\leq t+2),$
$R_{t+3i+2}-R_{t+3i+1}$ $(1\leq i\leq  p),$ $R_{i}-R_{t+3p+3}$
$(t+3p+4\leq i\leq t+3p+s_{p+2}+2)$, $\cdots$,
$R_{i}-R_{t+3p+s_{p+2}+\cdots+s_{k-1}+3}$
$(t+3p+4+\sum\limits_{j=p+2}^{k-1}s_{j}\leq i\leq  n),$ we get

$$\chi
(G,\lambda )={\left| {\begin{array}{*{20}{c}}
                 B_{11}& B_{12} \\
                  B_{21} &  B_{22}
                       \end{array}} \right|}$$

where

\begin{displaymath}
\footnotesize B_{11}=\left(
\begin{array}{ccccccccccccc}
                \lambda &0&0&0& \cdots &0&{ - 1}&{ - 2}&{ - 1}&\cdots&{ - 1}&{ - 2}&{ - 1}\\
                0&\lambda & { - t}& { - 1}& \cdots & { - 1}&{ - 1}& { - 2} &{ - 1}& \cdots&{ - 1}&{ - 2}&{ - 1}\\
                0& { - 1}&\lambda & {0}& \cdots & {0}&{ - 1}& { - 2} &{ - 1}& \cdots&{ - 1}&{ - 2}&{ - 1}\\
                0& { 0}& {0}&\lambda & \cdots & {0}&{ 0}& { 0} &{ 0}& \cdots&{ 0}&{ 0}&{ 0}\\
                \vdots & \vdots & \vdots & \vdots& \ddots & \vdots & \vdots & \vdots& \vdots &{}& \vdots &\vdots& \vdots\\
                0&{0}&0&0& \cdots &\lambda &{ 0}& { 0} &{0}& \cdots&{ 0}&{ 0}&{ 0}\\
                { - 1}& { - 1}&{ - t}&{ - 1}& \cdots &{ - 1}&\lambda&0&0 & \cdots &{ - 1}&{ - 2}&{ - 1}\\
                { - 1}& { - 1}&{ - t}&{ - 1}& \cdots &{ - 1}&0&{\lambda-1}&{ - 1} & \cdots &{ - 1}&{ - 2}&{ - 1}\\
                {0}& { 0}&{ 0}&{ 0}& \cdots &{ 0}&0&{ 0}&{\lambda+1} & \cdots &{ 0}&{ 0}&{ 0}\\
                \vdots & \vdots & \vdots & \vdots& {}& \vdots & \vdots& \vdots & \vdots& \ddots  &\vdots& \vdots &\vdots\\
                 { - 1}& { - 1}&{ - t}&{ - 1}& \cdots &{ - 1}&{ - 1}&{ - 2}&{ - 1}& \cdots&\lambda&0&0\\
                { - 1}& { - 1}&{ -  t}&{ - 1}& \cdots &{ - 1}&{ - 1}&{ - 2}&{ - 1}& \cdots &0&{\lambda-1}&{ - 1}\\
                { 0}& { 0}&{ 0}&{ 0}& \cdots &{ 0}&{ 0}&{ 0}&{ 0}&\cdots &0&{ 0}&{\lambda+1}\\
                \end{array} \right)_{(t+3p+2) \times (t+3p+2)},
\end{displaymath}

\begin{displaymath}
B_{21}=\left(
\begin{array}{ccccccccccccc}
                 { - 1}& { - 1}&{ -  t}&{ - 1}& \cdots &{ - 1}&{ - 1}&{ - 2}&{ - 1}& \cdots& { - 1}&{ - 2}&{ - 1}\\
                \vdots & \vdots & \vdots & \vdots& {} & \vdots & \vdots & \vdots& \vdots & {}& \vdots &\vdots& \vdots\\
                { 0}& { 0}&{ 0}&{ 0}& \cdots &{ 0}&{ 0}&{ 0}&{ 0}&\cdots & { 0}&{ 0}&{ 0}\\
                \vdots & \vdots & \vdots & \vdots& {} & \vdots & \vdots & \vdots& \vdots &{}& \vdots &\vdots& \vdots \\
                { - 1}& { - 1}&{ -  t}&{ - 1}& \cdots &{ - 1}&{ - 1}&{ - 2}&{ - 1}& \cdots& { - 1}&{ - 2}&{ - 1}\\
                \vdots & \vdots & \vdots & \vdots& {} & \vdots & \vdots & \vdots& \vdots &{}& \vdots &\vdots& \vdots \\
                { 0}& { 0}&{ 0}&{ 0}& \cdots &{ 0}&{ 0}&{ 0}&{ 0}&\cdots & { 0}&{ 0}&{ 0}\\
                \end{array} \right)_{(n-t-3p-2) \times (t+3p+2)},
\end{displaymath}

\begin{displaymath}
B_{12}=\left(
\begin{array}{ccccccc}
                { - s_{p+2}}&\cdots &{ - 1}&\cdots&{ - s_{k}}&\cdots &{ - 1}\\
                { - s_{p+2}}& \cdots &{ - 1}& \cdots&{- s_{k}}& \cdots &{ -1}\\
               { - s_{p+2}}& \cdots &{ - 1}& \cdots&{- s_{k}}& \cdots &{ -1}\\
               { 0}& \cdots &{0}& \cdots&{0}& \cdots &{ 0}\\
               \vdots &{}& \vdots & {} & \vdots &{}& \vdots \\
              { 0}& \cdots &{0}& \cdots&{ 0}& \cdots &{0}\\
               { -s_{p+2}}& \cdots &{ - 1}& \cdots&{ - s_{k}}& \cdots &{ -1}\\
               {  -s_{p+2}}& \cdots &{ - 1}& \cdots&{ - s_{k}}& \cdots &{ -1}\\
              { 0}& \cdots &{ 0}& \cdots&{ 0}& \cdots &{ 0}\\
                \vdots & {} &\vdots& {} &\vdots & {}& \vdots\\
                {-s_{p+2}}& \cdots &{ - 1}& \cdots&{ - s_{k}}& \cdots &{ -1}\\
                { -s_{p+2}}& \cdots &{ - 1}& \cdots&{ - s_{k}}& \cdots &{ -1}\\
               { 0}& \cdots &{ 0}& \cdots&{ 0}& \cdots &{ 0}\\
                \end{array} \right)_{(t+3p+2) \times (n-t-3p-2)},
\end{displaymath}

\begin{displaymath}
B_{22}=\left(
\begin{array}{ccccccc}
                \lambda& \cdots &0 & \cdots &{ -s_{k}}& \cdots &{ - 1}\\
                 \vdots & \ddots  &\vdots& {} & \vdots & {} & \vdots\\
               0& \cdots&\lambda & \cdots&{ 0}& \cdots &{ 0}\\
                \vdots & {} & \vdots & \ddots  &\vdots& {} & \vdots \\
               { - s_{p+2}}& \cdots &{ - 1}& \cdots &\lambda& \cdots &{ 0}\\
                 \vdots & {} & \vdots &{}& \vdots &  \ddots  & \vdots\\
               { 0}& \cdots &{ 0}& \cdots &0 & \cdots&\lambda\\
                 \end{array} \right)_{(n-t-3p-2) \times (n-t-3p-2)},
\end{displaymath}

Then by Laplace expansion along Rows $i$ $(4\leq i\leq t+3p+2, i\neq
t+3l+3,t+3l+4, 0\leq l\leq p-1)$ and Rows $i$ $(t+3p+4\leq i \leq n,
i\neq t+3p+3, t+3p+s_{p+2}+3,\cdots,
t+3p+\sum\limits_{j=p+2}^{k-1}s_{j}+3),$ we have
 $$\chi(G,\lambda)=\lambda^{\eta +t-1}(\lambda+1)^{p}Q(\lambda),$$
where $\eta=\sum\limits_{i=p+2}^{k} s_i-k+p+1$ and

$Q(\lambda)=$
$${\left| {\begin{array}{*{20}{c}}
                \lambda &0&0&{ - 1}&{ - 2}&{ - 1}&{ - 2}&\cdots&{ - 1}&{ - 2}&{ - s_{p+2}}&\cdots&{ - s_{k}}\\
                0&\lambda & { - t}& { - 1}& { - 2}&{ - 1}& { - 2} & \cdots&{ - 1}&{ - 2}&{ - s_{p+2}}& \cdots&{- s_{k}}\\
                0& { - 1}&\lambda &{ - 1}&{ - 2}&{ - 1}& { - 2}& \cdots&{ - 1}&{ - 2}&{ - s_{p+2}}& \cdots&{- s_{k}}\\
                { - 1}& { - 1}&{ - t}&\lambda&0 &{ - 1}&{ - 2}& \cdots &{ - 1}&{ - 2}&{ -s_{p+2}}& \cdots&{ - s_{k}}\\
                { - 1}& { - 1}&{ - t}&0&{\lambda-1}&{ - 1}&{ - 2} & \cdots &{ - 1}&{ - 2}&{  -s_{p+2}}& \cdots&{ - s_{k}}\\
                { - 1}& { - 1}&{ - t}&{ - 1} &{ - 2}&{ \lambda}&{ 0}& \cdots&{ - 1}&{-2}  &{-s_{p+2}}& \cdots&{ - s_{k}}\\
                { - 1}& { - 1}&{ -  t}&{ - 1} &{ - 2}&{ 0}&{ \lambda-1}& \cdots &{ - 1}&{-2}&{ -s_{p+2}}& \cdots&{ - s_{k}}\\
                \vdots & \vdots & \vdots & \vdots & \vdots & \vdots & \vdots& \ddots& \vdots &\vdots& \vdots & {}  &\vdots \\
                { - 1}& { - 1}&{ -  t}&{ - 1}&{ - 2}&{ - 1}&{ - 2}& \cdots& { \lambda}&{ 0}&{ -s_{p+2}}& \cdots &{ - s_{k}}\\
                { - 1}& { - 1}&{ -  t}&{ - 1} &{ - 2}&{ - 1}&{ - 2}& \cdots& { 0}&{\lambda - 1}&{ -s_{p+2}}& \cdots &{ - s_{k}}\\
                { - 1}& { - 1}&{ -  t}&{ - 1} &{ - 2}&{ - 1}&{ - 2}& \cdots& { - 1}&{ - 2}&{\lambda}& \cdots &{ - s_{k}}\\
                \vdots & \vdots & \vdots & \vdots & \vdots & \vdots & \vdots& {}& \vdots &\vdots& \vdots & \ddots  &\vdots \\
                { - 1}& { - 1}&{ -  t}&{ - 1} &{ -2}&{ - 1}&{ - 2}& \cdots& { - 1}&{ - 2}&{ -s_{p+2}}& \cdots &{\lambda}
        \end{array}} \right|}$$
In fact $Q(\lambda)=$
$${\left| {\begin{array}{*{20}{c}}
                {  1}&0 &0&0&{0}&{ 0}&{ 0}&{ 0}&\cdots&{ 0}&{0}&{ 0}&\cdots&{0}\\
                {  1}&\lambda &0&0&{ - 1}&{ - 2}&{ - 1}&{ - 2}&\cdots&{ - 1}&{ - 2}&{ - s_{p+2}}&\cdots&{ - s_{k}}\\
                {  1}&0&\lambda & { - t}& { - 1}& { - 2}&{ - 1}& { - 2} & \cdots&{ - 1}&{ - 2}&{ - s_{p+2}}& \cdots&{- s_{k}}\\
                {  1}&0& { - 1}&\lambda &{ - 1}&{ - 2}&{ - 1}& { - 2}& \cdots&{ - 1}&{ - 2}&{ - s_{p+2}}& \cdots&{- s_{k}}\\
                {  1}&{ - 1}& { - 1}&{ - t}&\lambda&0 &{ - 1}&{ - 2}& \cdots &{ - 1}&{ - 2}&{ -s_{p+2}}& \cdots&{ - s_{k}}\\
                {  1}&{ - 1}& { - 1}&{ - t}&0&{\lambda-1}&{ - 1}&{ - 2} & \cdots &{ - 1}&{ - 2}&{  -s_{p+2}}& \cdots&{ - s_{k}}\\
                {  1}&{ - 1}& { - 1}&{ - t}&{ - 1} &{ - 2}&{ \lambda}&{ 0}& \cdots&{ - 1}&{-2}  &{-s_{p+2}}& \cdots&{ - s_{k}}\\
                {  1}&{ - 1}& { - 1}&{ -  t}&{ - 1} &{ - 2}&{ 0}&{ \lambda-1}& \cdots &{ - 1}&{-2}&{ -s_{p+2}}& \cdots&{ - s_{k}}\\
                {  \vdots}&\vdots & \vdots & \vdots & \vdots & \vdots & \vdots & \vdots& \ddots& \vdots &\vdots& \vdots & {}  &\vdots \\
                {  1}&{ - 1}& { - 1}&{ -  t}&{ - 1}&{ - 2}&{ - 1}&{ - 2}& \cdots& { \lambda}&{ 0}&{ -s_{p+2}}& \cdots &{ - s_{k}}\\
                {  1}&{ - 1}& { - 1}&{ -  t}&{ - 1} &{ - 2}&{ - 1}&{ - 2}& \cdots& { 0}&{\lambda - 1}&{ -s_{p+2}}& \cdots &{ - s_{k}}\\
                {  1}&{ - 1}& { - 1}&{ -  t}&{ - 1} &{ - 2}&{ - 1}&{ - 2}& \cdots& { - 1}&{ - 2}&{\lambda}& \cdots &{ - s_{k}}\\
                {  \vdots}&\vdots & \vdots & \vdots & \vdots & \vdots & \vdots & \vdots& {}& \vdots &\vdots& \vdots & \ddots  &\vdots \\
                {  1}&{ - 1}& { - 1}&{ -  t}&{ - 1} &{ -2}&{ - 1}&{ - 2}& \cdots& { - 1}&{ - 2}&{ -s_{p+2}}& \cdots &{\lambda}
        \end{array}} \right|}$$

By $C_{i}+C_{1}$ $(i=2,3),$
 $C_{4}+tC_{1},$
 $C_{2l+3}+C_{1}$ $(1\leq l\leq p),$ $C_{2l+4}+2C_{1}$
$(1\leq l\leq p)$ and $C_{i}+s_{i-(p+3)}C_{1}$ $(2p+5\leq i\leq
p+k+3)$, the determinant becomes

$$\footnotesize{\left| {\begin{array}{*{20}{c}}
                {  1}&1 &1&t&{1}&{ 2}&{ 1}&{ 2}&\cdots&{ 1}&{2}&{s_{p+2}}&\cdots&{s_{k}}\\
                {  1}&\lambda+1 &1&t&{ 0}&{ 0}&{ 0}&{ 0}&\cdots&{ 0}&{ 0}&{ 0}&\cdots&{ 0}\\
                {  1}&1&\lambda+1 & { 0}& { 0}& { 0}&{ 0}& { 0} & \cdots&{ 0}&{0}&{ 0}& \cdots&{0}\\
                {  1}&1& { 0}&\lambda+t &{ 0}&{ 0}&{ 0}& { 0}& \cdots&{ 0}&{ 0}&{ 0}& \cdots&{0}\\
                {  1}&{ 0}& {0}&{ 0}&\lambda+1&2 &{ 0}&{ 0}& \cdots &{ 0}&{ 0}&{0}& \cdots&{ 0}\\
                {  1}&{ 0}& { 0}&{0}&1&{\lambda+1}&{ 0}&{ 0} & \cdots &{ 0}&{ 0}&{ 0}& \cdots&{ 0}\\
                {  1}&{ 0}& {0}&{ 0}&{ 0} &{ 0}&{ \lambda+1}&{2}& \cdots&{ 0}&{0}  &{0}& \cdots&{ 0}\\
                {  1}&{ 0}& { 0}&{ 0}&{ 0} &{ 0}&{ 1}&{ \lambda+1}& \cdots &{ 0}&{0}&{ 0}& \cdots&{ 0}\\
                \vdots&\vdots & \vdots & \vdots & \vdots & \vdots & \vdots & \vdots& \ddots& \vdots &\vdots& \vdots & {}  &\vdots \\
                {  1}&{ 0}& { 0}&{ 0}&{ 0}&{ 0}&{ 0}&{ 0}& \cdots& { \lambda+1}&{2}&{ 0}& \cdots &{ 0}\\
                {  1}&{ 0}& {0}&{ 0}&{0} &{ 0}&{ 0}&{ 0}& \cdots& { 1}&{\lambda+ 1}&{ 0}& \cdots &{ 0}\\
                {  1}&{ 0}& {0}&{ 0}&{ 0} &{ 0}&{ 0}&{ 0}& \cdots& { 0}&{ 0}&{\lambda+s_{p+2}}& \cdots &{ 0}\\
                \vdots&\vdots & \vdots & \vdots & \vdots & \vdots & \vdots & \vdots&  {} & \vdots &\vdots& \vdots & \ddots  &\vdots \\
                {  1}&{ 0}& { 0}&{ 0}&{ 0} &{ 0}&{ 0}&{0}& \cdots& { 0}&{ 0}&{ 0}& \cdots &{\lambda+s_{k}}
        \end{array}} \right|}$$

By $C_{5+2l}-C_{5}$ $(1\leq l\leq p-1),$ $C_{6+2l}-C_{6}$ $(1\leq
l\leq p-1),$ row operations $R_{5}+R_{5+2l}$ $(1\leq l\leq p-1),$
$R_{6}+R_{6+2l}$ $(1\leq l\leq p-1),$
$R_{1}-\sum\limits_{i=p+2}^{k}\frac{s_{i}}{\lambda+s_{i}}R_{i+p+3}$
$( \lambda\neq -s_{i})$ and Laplace expansion, we get $Q(\lambda)=$
$$\left| {\begin{array}{*{6}{c}}
                1-\sum\limits_{i=p+2}^k\frac{s_i}{\lambda+s_i} &1&1&t&1&2\\
                1&\lambda+1 &1&t&0&0\\
                1&1&\lambda+1 &0&0&0\\
                1&1&0&\lambda+t &0&0\\
                p&0&0&0&\lambda+1&2\\
                p&0&0&0&1&\lambda+1
        \end{array}} \right|
\left| {\begin{array}{*{2}{c}}
                \lambda+1&2\\
                1&\lambda+1
        \end{array}} \right|^{p-1}\prod_{j=p+2}^k(\lambda+s_j).$$

9. (For the proof of Lemma 2.19(ii)) Let $G=(K_1\cup K_{1,3})\vee
(K_1\cup K_{1,2})\vee{\overline K_{s_3}}\vee\cdots\vee{\overline
K_{s_{k}}}$. Then $\chi(G,\lambda)=$
$$\footnotesize{\left| {\begin{array}{*{20}{c}}
                \lambda &0&0& 0 &0&{ - 1}&{ - 1}& { - 1} &{ - 1}&{ - 1}&{ - 1}& \cdots &{ - 1}& \cdots &{ - 1}&{ - 1}& \cdots&{ - 1}\\
                0&\lambda &{ - 1}& { - 1} &{ - 1}&{ { - 1}}&{ - 1}& { - 1} &{ - 1}&{ - 1}&{ - 1}& \cdots &{ - 1}& \cdots &{ - 1} &{ - 1}& \cdots&{ - 1}\\
                0&{ - 1}&\lambda & 0 &0&{ - 1}&{ - 1}& { - 1} &{ - 1}&{ - 1}&{ - 1}& \cdots &{ - 1}& \cdots &{ - 1} &{ - 1}& \cdots &{ - 1}\\
                0 & { - 1} & 0 & \lambda & 0 & { - 1} & { - 1} &  { - 1}&  { - 1}  &{ - 1}&{ - 1}& \cdots &{ - 1}& \cdots &{ - 1}&{ - 1}& \cdots&{ - 1}\\
                0&{ - 1}&0&0 &\lambda &{ - 1}&{ - 1}& { - 1} &{ - 1}&{ - 1}&{ - 1}& \cdots &{ - 1}& \cdots &{ - 1}&{ - 1}& \cdots&{ - 1}\\
                { - 1}&{ - 1} &{ - 1}& { - 1} &{ - 1}&{ \lambda}&{ 0}& 0 &{ 0}&{ - 1}&{ - 1}& \cdots &{ - 1}& \cdots &{ - 1}&{ - 1}& \cdots&{ - 1}\\
                { - 1}&{ - 1}&{ - 1} & { - 1} &{ - 1}&0 &\lambda & { - 1} &{ - 1}&{ - 1}&{ - 1}& \cdots &{ - 1}& \cdots &{ - 1}&{ - 1}& \cdots&{ - 1}\\
                { - 1} & { - 1} & { - 1} & { - 1}& { - 1} & 0 &{ - 1} & { \lambda} & 0 &{ - 1}&{ - 1}& \cdots &{ - 1}& \cdots &{ - 1}&{ - 1}& \cdots&{ - 1}\\
                { - 1}&{ - 1}&{ - 1}&{ - 1} &{ - 1} &0&{ - 1}& 0 & \lambda &{ - 1}&{ - 1}& \cdots &{ - 1}& \cdots &{ - 1}&{ - 1}& \cdots&{ - 1}\\
                { -1}&{ - 1} &{ - 1}&{ - 1} &{ - 1}&{ - 1}&{- 1}& { - 1} &{ - 1}&{ \lambda}&0& \cdots &{ 0}& \cdots  &{ - 1}&{ - 1} & \cdots  &{ - 1}\\
                { -1}&{ - 1} &{ - 1}&{ - 1} &{ - 1}&{ - 1}&{- 1}& { - 1} &{ - 1}&0&{ \lambda}& \cdots&{ 0} & \cdots &{ - 1}&{ - 1}& \cdots  &{ - 1}\\
                \vdots & \vdots & \vdots & \vdots  & \vdots & \vdots &\vdots & \vdots & \vdots& \vdots & \vdots & \ddots  & \vdots & {} & \vdots & \vdots &{}& \vdots \\
                { -1}&{ -1}&{ -1}& { -1} &{ -1} &{ -1}&{ -1}& { -1} & { -1} &0&0& \cdots & \lambda& \cdots &{ -1}&{ -1}&\cdots&{ -1}\\
                \vdots & \vdots & \vdots & \vdots & \vdots & \vdots &\vdots & \vdots & \vdots & \vdots& \vdots & {} & \vdots & \ddots & \vdots & \vdots &{}& \vdots \\
                { -1}&{ -1} &{ - 1}& { -1} &{ - 1}&{ -1}&{- 1}& { -1} &{ - 1}&{ -1}&{-1}&\cdots &{-1}& \cdots &{ \lambda} &{ 0}& \cdots&{ 0}\\
                { -1}&{ -1}&{ -1}& { -1} &{ -1}&{ -1} &{ -1}& { -1} &{ -1}&{ -1}&{ -1}& \cdots &{ -1}& \cdots &0 &\lambda& \cdots &0\\
                \vdots & \vdots & \vdots & \vdots & \vdots & \vdots &\vdots & \vdots & \vdots & \vdots& \vdots & {} & \vdots & {} & \vdots & \vdots &\ddots& \vdots \\
                { -1}&{ -1}&{ -1}& { -1} &{ -1} &{ -1}&{ -1}& { -1} & { -1} &{ -1}&{ -1}& \cdots & { -1}& \cdots &0&0 &\cdots& \lambda\\
        \end{array}} \right|_{n \times n}}$$

By $C_{3}+\sum\limits_{i=4}^{5}C_{i}$, $C_{8}+C_{9},$
$C_{10}+\sum\limits_{i=11}^{s_{3}+9}C_{i}$,
$C_{s_{3}+10}+\sum\limits_{i=s_{3}+11}^{s_{3}+s_{4}+9}C_{i}$,
 $\cdots,$
$C_{s_{3}+\cdots+s_{k-1}+10}+\sum\limits_{i=11+\sum\limits_{j=3}^{k-1}s_{j}}^{n}C_{i}$,
 $R_{i}-R_{3}$
$(i=4,5),$ $R_{9}-R_{8},$ $R_{i}-R_{10}$ $(11\leq i\leq s_{3}+9)$,
$R_{i}-R_{s_{3}+10}$ $(s_{3}+11\leq i\leq s_{3}+s_{4}+9)$,$\cdots$,
$R_{i}-R_{s_{3}+\cdots+s_{k-1}+10}$
$(11+\sum\limits_{j=3}^{k-1}s_{j}\leq i\leq n)$ and then by Laplace
expansion, we get $\chi(G,\lambda)=\lambda^{\xi+3}R_{1}(\lambda),$
where $\xi=\sum\limits_{i=3}^{k}s_i-k+2$ and
$$R_{1}(\lambda)={\left| {\begin{array}{*{20}{c}}
                \lambda &0&0&{ - 1} &{ - 1}&{ - 2}&{ - s_{3}}&{ -s_{4}}& \cdots&{ -s_{k}}\\
                0&\lambda &{ - 3}& { - 1} &{ - 1}&{ -2}&{ - s_{3}}&{ -s_{4}}& \cdots&{ -s_{k}}\\
                0&{ - 1}&\lambda & { - 1} &{ - 1}&{ -2}&{ - s_{3}}&{ -s_{4}}& \cdots&{ -s_{k}}\\
                { - 1} & { - 1} & { - 3} & \lambda & 0 & { 0} &{ - s_{3}}&{ -s_{4}}& \cdots&{ -s_{k}}\\
                { - 1}&{ - 1}&{ -3}&0 &\lambda &{ - 2}&{ - s_{3}}&{ -s_{4}}& \cdots&{ -s_{k}}\\
                { - 1}&{ - 1} &{ -3}& 0 &{ - 1}&{ \lambda}&{ - s_{3}}&{ -s_{4}}& \cdots&{ -s_{k}}\\
                { - 1}&{ - 1}&{ - 3} & { - 1} &{ - 1}&{ - 2} &{ \lambda}&{ -s_{4}}& \cdots&{ -s_{k}}\\
                { - 1} & { - 1} & { - 3} & { - 1}& { - 1} & { - 2}  &{ - s_{3}}&{ \lambda}& \cdots&{ -s_{k}}\\
                \vdots & \vdots & \vdots & \vdots  & \vdots & \vdots  & \vdots &\vdots&\ddots& \vdots\\
                { -1}&{ - 1} &{ - 3}&{ - 1} &{ - 1}&{ - 2}&{ - s_{3}}&{ -s_{4}}& \cdots&{\lambda}
        \end{array}} \right|_{(k+4) \times (k+4)}}.$$

In fact
$$R_{1}(\lambda)={\left| {\begin{array}{*{20}{c}}
                1 &0&0&0&0 &0&0&0&0& \cdots&0\\
                1 &\lambda &0&0&{ - 1} &{ - 1}&{ - 2}&{ - s_{3}}&{ -s_{4}}& \cdots&{ -s_{k}}\\
                1 &0&\lambda &{ - 3}& { - 1} &{ - 1}&{ -2}&{ - s_{3}}&{ -s_{4}}& \cdots&{ -s_{k}}\\
                1 &0&{ - 1}&\lambda & { - 1} &{ - 1}&{ -2}&{ - s_{3}}&{ -s_{4}}& \cdots&{ -s_{k}}\\
                1 &{ - 1} & { - 1} & { - 3} & \lambda & 0 & { 0} &{ - s_{3}}&{ -s_{4}}& \cdots&{ -s_{k}}\\
                1 &{ - 1}&{ - 1}&{ -3}&0 &\lambda &{ - 2}&{ - s_{3}}&{ -s_{4}}& \cdots&{ -s_{k}}\\
                1 &{ - 1}&{ - 1} &{ -3}& 0 &{ - 1}&{ \lambda}&{ - s_{3}}&{ -s_{4}}& \cdots&{ -s_{k}}\\
                1 &{ - 1}&{ - 1}&{ - 3} & { - 1} &{ - 1}&{ - 2} &{ \lambda}&{ -s_{4}}& \cdots&{ -s_{k}}\\
                1 &{ - 1} & { - 1} & { - 3} & { - 1}& { - 1} & { - 2}  &{ - s_{3}}&{ \lambda}& \cdots&{ -s_{k}}\\
                \vdots &\vdots & \vdots & \vdots & \vdots  & \vdots & \vdots  & \vdots &\vdots&\ddots& \vdots\\
                1 &{ -1}&{ - 1} &{ - 3}&{ - 1} &{ - 1}&{ - 2}&{ - s_{3}}&{ -s_{4}}& \cdots&{\lambda}
        \end{array}} \right|_{(k+5) \times (k+5)}}$$

Similar to Appendix 7, by $C_{i}+C_{1}$
$(2\leq i\leq 6, i\neq 4),$ $C_{4}+3C_{1},$
 $C_{7}+2C_{1},$
 $C_{l+5}+s_{l}C_{1}$ $(3\leq l\leq k),$
 $R_{1}-\sum\limits_{l=3}^{k}\frac{s_{l}}{\lambda+s_{l}}R_{l+5}$
$(\lambda\neq-s_{l})$ and then by Laplace expansion, we get
$$R_{1}(\lambda)=\prod_{i=3}^k(\lambda+s_i)R(\lambda),$$
and

$$R(\lambda)=\left| {\begin{array}{*{7}{c}}
                1-\sum\limits_{i=3}^k\frac{s_i}{\lambda+s_i} &1&1&3&1&1&2\\
                1&\lambda+1 &1&3&0&0&0\\
                1&1&\lambda+1 &0&0&0&0\\
                1&1&0&\lambda+3 &0&0&0\\
                1&0&0&0&\lambda+1&1&2\\
                1&0&0&0&1&\lambda+1&0\\
                1&0&0&0&1&0&\lambda+2
        \end{array}} \right|.$$

Thus
$$\chi(G,\lambda)=\lambda^{\xi+3}\prod_{i=3}^k(\lambda+s_i)R(\lambda),$$
where $\xi=\sum\limits_{i=3}^{k}s_i-k+2.$

10. (For the proof of Lemma 2.19(iii)) Let $G=(K_1\cup K_{1,3})\vee
(K_1\cup K_{1,2})\vee (K_1\cup K_{1,1})\vee{\overline K_{s_4}}.$
Then $\chi(G,\lambda)=$
$${\left| {\begin{array}{*{20}{c}}
                \lambda &0&0& 0 &0&{ - 1}&{ - 1}& { - 1} &{ - 1}&{ - 1}&{ - 1}&{ - 1}&{ - 1}& \cdots &{ - 1}\\
                0&\lambda &{ - 1}& { - 1} &{ - 1}&{ { - 1}}&{ - 1}& { - 1} &{ - 1}&{ - 1}&{ - 1}&{ - 1}&{ - 1}& \cdots &{ - 1}\\
                0&{ - 1}&\lambda & 0 &0&{ - 1}&{ - 1}& { - 1} &{ - 1}&{ - 1}&{ - 1}&{ - 1}&{ - 1}& \cdots &{ - 1}\\
                0 & { - 1} & 0 & \lambda & 0 & { - 1} & { - 1} &  { - 1}&  { - 1}  &{ - 1}&{ - 1}&{ - 1}&{ - 1}& \cdots &{ - 1}\\
                0&{ - 1}&0&0 &\lambda &{ - 1}&{ - 1}& { - 1} &{ - 1}&{ - 1}&{ - 1}&{ - 1}&{ - 1}& \cdots &{ - 1}\\
                { - 1}&{ - 1} &{ - 1}& { - 1} &{ - 1}&{ \lambda}&{ 0}& 0 &{ 0}&{ - 1}&{ - 1}&{ - 1}&{ - 1}& \cdots &{ - 1}\\
                { - 1}&{ - 1}&{ - 1} & { - 1} &{ - 1}&0 &\lambda & { - 1} &{ - 1}&{ - 1}&{ - 1}&{ - 1}&{ - 1}& \cdots &{ - 1}\\
                { - 1} & { - 1} & { - 1} & { - 1}& { - 1} & 0 &{ - 1} & { \lambda} & 0 &{ - 1}&{ - 1}&{ - 1}&{ - 1}& \cdots &{ - 1}\\
                { - 1}&{ - 1}&{ - 1}&{ - 1} &{ - 1} &0&{ - 1}& 0 & \lambda &{ - 1}&{ - 1}&{ - 1}&{ - 1}& \cdots &{ - 1}\\
                { - 1}&{ - 1} &{ - 1}& { - 1} &{ - 1}&{ - 1}&{ - 1}&{ - 1}&{ - 1}&{ \lambda}&{ 0}& 0 &{ -1}& \cdots &{ - 1}\\
                { - 1}&{ - 1}&{ - 1} & { - 1} &{ - 1}&{ - 1}&{ - 1}&{ - 1}&{ - 1}&0 &{\lambda} & { - 1} &{ - 1}& \cdots &{ - 1}\\
                { - 1} & { - 1} & { - 1} & { - 1}& { - 1} &{ - 1}&{ - 1}&{ - 1}&{ - 1}& 0 &{ - 1} & { \lambda} &{ - 1} & \cdots &{ - 1}\\
                { - 1}&{ - 1}&{ - 1}&{ - 1} &{ - 1} &-1&{ - 1}&-1 & { - 1} &{ - 1}&{ - 1}&{ - 1}&{\lambda}& \cdots &{ 0}\\
                \vdots & \vdots & \vdots & \vdots  & \vdots & \vdots &\vdots & \vdots & \vdots & \vdots & \vdots & \vdots & \vdots & \ddots  & \vdots \\
                { -1}&{ -1}&{ -1}& { -1} &{ -1} &{ -1}&{ -1}& { -1} &{ - 1}&{ - 1}& { -1} &{ - 1}&0& \cdots & \lambda\\
                       \end{array}} \right|_{n \times n}}$$

By $C_{3}+\sum\limits_{i=4}^{5}C_{i}$, $C_{8}+C_{9},$
$C_{11}+C_{12},$ $C_{13}+\sum\limits_{i=14}^{n}C_{i}$, $\cdots,$
$R_{i}-R_{3}$ $(i=4,5),$ $R_{9}-R_{8},$ $R_{12}-R_{11}$,
$R_{i}-R_{13}$ $(14\leq i\leq n)$ and Laplace expansion, we have
$$\chi(G,\lambda)=\lambda^{s_{4}+2}(\lambda+1)S(\lambda),$$

where $$S(\lambda)={\left| {\begin{array}{*{20}{c}}
                \lambda &0&0&{ - 1} &{ - 1}&{ - 2}&{ -1}&{ -2}&{ -s_{4}}\\
                0&\lambda &{ - 3}& { - 1} &{ - 1}&{ -2}&{ - 1}&{ -2}&{ -s_{4}}\\
                0&{ - 1}&\lambda & { - 1} &{ - 1}&{ -2}&{ -1}&{ -2}&{ -s_{4}}\\
                { - 1} & { - 1} & { - 3} & \lambda & 0 & { 0} &{ -1}&{ -2}&{ -s_{4}}\\
                { - 1}&{ - 1}&{ -3}&0 &\lambda &{ - 2}&{ -1}&{ -2}&{ -s_{4}}\\
                { - 1}&{ - 1} &{ -3}& 0 &{ - 1}&{ \lambda}&{ - 1}&{ -2}&{ -s_{4}}\\
                { - 1}&{ - 1}&{ - 3} & { - 1} &{ - 1}&{ - 2} &{ \lambda}&{ 0}&{ -s_{4}}\\
                { - 1} & { - 1} & { - 3} & { - 1}& { - 1} & { - 2}  &{ 0}&{ \lambda-1}&{ -s_{4}}\\
                { -1}&{ - 1} &{ - 3}&{ - 1} &{ - 1}&{ - 2}&{ -1}&{ -2}&{\lambda}
        \end{array}} \right|}$$
$$={\left| {\begin{array}{*{20}{c}}
                {  1}&0 &0&0&0 &0&0&0&0&0\\
                {  1}&\lambda &0&0&{ - 1} &{ - 1}&{ - 2}&{ -1}&{ -2}&{ -s_{4}}\\
                {  1}&0&\lambda &{ - 3}& { - 1} &{ - 1}&{ -2}&{ - 1}&{ -2}&{ -s_{4}}\\
                {  1}&0&{ - 1}&\lambda & { - 1} &{ - 1}&{ -2}&{ -1}&{ -2}&{ -s_{4}}\\
                {  1}&{ - 1} & { - 1} & { - 3} & \lambda & 0 & { 0} &{ -1}&{ -2}&{ -s_{4}}\\
                {  1}&{ - 1}&{ - 1}&{ -3}&0 &\lambda &{ - 2}&{ -1}&{ -2}&{ -s_{4}}\\
                {  1}&{ - 1}&{ - 1} &{ -3}& 0 &{ - 1}&{ \lambda}&{ - 1}&{ -2}&{ -s_{4}}\\
                {  1}&{ - 1}&{ - 1}&{ - 3} & { - 1} &{ - 1}&{ - 2} &{ \lambda}&{ 0}&{ -s_{4}}\\
                {  1}&{ - 1} & { - 1} & { - 3} & { - 1}& { - 1} & { - 2}  &{ 0}&{ \lambda-1}&{ -s_{4}}\\
                {  1}&{ -1}&{ - 1} &{ - 3}&{ - 1} &{ - 1}&{ - 2}&{ -1}&{ -2}&{\lambda}
        \end{array}} \right|}$$

$$=\left| {\begin{array}{*{10}{c}}
                1&0&0&0&0&0&0&0&0&s_4\\
                1&\lambda&0&0&-1&-1&-2&-1&-2&0\\
                1&0&\lambda &-3&-1&-1&-2&-1&-2&0\\
                1&0&-1&\lambda &-1&-1&-2&-1&-2&0\\
                1&-1&-1&-3&\lambda&0&0&-1&-2&0\\
                1&-1&-1&-3&0&\lambda&-2&-1&-2&0\\
                1&-1&-1&-3&0&-1&\lambda &-1&-2&0\\
                1&-1&-1&-3&-1&-1&-2&\lambda &0&0\\
                1&-1&-1&-3&-1&-1&-2&0&\lambda-1 &0\\
                1&-1&-1&-3&-1&-1&-2&-1&-2&\lambda+s_4
        \end{array}} \right|.$$

\end{document}